\documentclass[11pt]{article}
\usepackage{amsfonts,bm}
\usepackage{amssymb}
\usepackage{latexsym}
\usepackage{mathrsfs}
\usepackage{multirow}
\usepackage{multicol}
\usepackage{graphicx}
\usepackage{rotating}
\usepackage{subfigure}
\usepackage{color}
\usepackage{algorithm}
\usepackage{algorithmic}
\usepackage{cite}
\usepackage{amsthm}
\usepackage[cmex10]{amsmath}
\usepackage{amsmath}
\usepackage{bm}
\usepackage{url}
\usepackage{hyperref}
\usepackage{booktabs}
\usepackage{float}
\usepackage[pagewise]{lineno}
\newtheorem*{definition[theorem]}{Definition}
\newtheorem{theorem}{Theorem}[section]

\newtheorem{lemma}{Lemma}[section]
\newtheorem{example}[theorem]{Example}

\numberwithin{equation}{section}

\date{}
\begin{document}
\title{A fast second-order accurate difference schemes for time distributed-order and Riesz space fractional diffusion equations}

\author{Huan-Yan Jian\thanks{\textit{E-mail address:} uestc\_hyjian@sina.com},
Ting-Zhu Huang \thanks{Corresponding author. \textit{E-mail address:}
tingzhuhuang@126.com. Tel.: 86-28-61831608},
Xi-Le Zhao \thanks{\textit{E-mail address:} xlzhao122003@163.com},
Yong-Liang Zhao \thanks{\textit{E-mail address:}  uestc\_ylzhao@sina.com}\\
{\small{\it School of Mathematical Sciences,}}\\
{\small{\it University of Electronic Science and Technology of China,}}\\
{\small{\it Chengdu, Sichuan 611731, P.R. China}}\\
}
\maketitle
\begin{center}
\textbf{Abstract}
\end{center}
\quad The aim of this paper is to develop fast second-order accurate difference schemes for solving one- and
 two-dimensional time distributed-order and Riesz space fractional diffusion equations. We
 adopt the same measures for one- and two-dimensional problems as follows: we first transform the time distributed-order
 fractional diffusion problem into the multi-term time-space fractional diffusion problem with the composite trapezoid
 formula. Then, we propose a second-order accurate difference scheme based on the interpolation approximation on a special point
 to solve the resultant problem. Meanwhile, the unconditional stability and convergence of the new difference scheme in $L_2$-norm are proved. Furthermore, we find that the discretizations lead to a series of Toeplitz systems which can be efficiently solved by Krylov subspace methods with suitable circulant preconditioners. Finally, numerical results are presented to show the effectiveness of the proposed difference methods and demonstrate the fast convergence of our preconditioned Krylov subspace
 methods.


\section{Introduction}

\quad Fractional diffusion equations (FDEs) have recently attracted considerable attention and interest due to its
wide applications \cite{scher1975anomalous,schneider1989fractional,gao2015some,Zhao2013Total}.
Specifically, the time-fractional anomalous diffusion equation has become the focus of intensive investigations from both theoretical and practical perspectives \cite{henry2006anomalous,zhuang2008new,sun2009variable,magin2008anomalous}.

Recently, the time-fractional anomalous diffusion equation with a single-term temporal derivative has been
discussed and studied in \cite{cui2009compact,li2014higher}. In \cite{langlands2006solution,liu2009numerical,liu2011finite}, the two-term time fractional diffusion equation was reported for describing processes that tend to be less anomalous. More generalized models were also developed as multi-term FDEs \cite{zheng2016high}, where several
fractional derivatives were simultaneously involved. To solve such problems, multiple numerical approaches
\cite{gu2015strang,luo2016quadratic,Li2017meng} have emerged, among which the finite difference method has grown popular \cite{chen2007fourier,lin2007finite,gu2016fast,chen2010numerical,cui2013convergence,zhang2014error}.

Although the single-term and multi-term FDEs are used extensively in many scientific fields, it is difficult for them to describe the non-Markovian processes for continuous time-scale distributions. Therefore, the time distributed-order FDEs \cite{chechkin2002retarding} began to attract the attention of researchers. It can be considered as a generalization of the multi-term FDEs and has been found to be an important tool for modeling ultraslow diffusion processes, accelerating sub-diffusion and other forms of strong anomaly \cite{chechkin2002retarding,kochubei2008distributed,gao2015two,hu2016implicit}.
The numerical method presented in \cite{katsikadelis2014numerical} for solving the distributed-order FDE consists of: (a) approximation of the integral with a finite sum using a simple quadrature rule so that the distributed order FDE is converted into a multi-term FDE and (b) development of a numerical method to solve the resultant multi-term FDE.
Such idea is essential for numerically solving the distributed-order FDEs and should be studied extensively. However, as far as we know, only a few algorithms have been developed to solve the distribution-order FDEs based on this idea.
Ye et al. \cite{ye2014numerical} proposed an implicit difference method for the time distributed-order and Riesz space FDEs on bounded domains and proved the difference method was unconditionally stable and convergent.
An implicit numerical method of a new time distributed-order and two-sided space-fractional advection-dispersion equation were constructed by Hu et al. \cite{hu2016implicit}.
In \cite{gao2015two}, Gao et al. explored two alternating direction implicit difference schemes with the unconditional stability and convergence analysis for solving the two-dimensional distributed-order FDEs.
Bu et al. \cite{bu2017finite} introduced the finite difference method for a class of distributed-order time FDEs on bounded domains.
In addition, most of these numerical approaches have no complete theoretical analysis of stability and convergence, especially for the time distribution-order and spatial FDEs, see \cite{katsikadelis2014numerical,liu2013numerical} for details.

In the current paper, inspired by the above observations, we consider effective numerical methods for the
following new time distributed-order and Riesz space FDEs (TDRFDEs):
\begin{align}
&{}D_{t}^{\omega(\alpha)}u(\textbf{x},t)=Au(\textbf{x},t)+f(\textbf{x},t),\quad \textbf{x}\in\Omega,~0<t\leq T,&\label{1.1}\\
&u(\textbf{x},t)|_{\textbf{x}\in\partial\Omega}=0,\quad 0\leq t\leq T,&\label{1.2}\\
&u(\textbf{x},0)=\phi(\textbf{x}),\quad \textbf{x}\in\Omega,\label{1.3}&
\end{align}
where $\alpha\in(0,1]$, $A$ is an operator and the function $f(\textbf{x},t)$ is the source term with sufficient smoothness. In particular, if $\Omega=(x_L,x_R)\subset{\mathbb{R}}$, then
$$A=K\frac{\partial^{\beta}}{\partial|x|^{\beta}},\quad K>0,\quad f(\textbf{x},t)=f(x,t);$$
if $\Omega=(x_L,x_R)\times(y_L,y_R)\subset{\mathbb{R}^2}$, then
$$A=K_1\frac{\partial^{\beta}}{\partial|x|^{\beta}}+K_2\frac{\partial^{\gamma}}{\partial|y|^{\gamma}},\quad K_1,K_2>0,\quad f(\textbf{x},t)=f(x,y,t),$$
where $\beta,\gamma\in(1,2]$, and the $\frac{\partial^{\beta}}{\partial\mid x\mid^{\beta}}$ is the Riesz fractional derivative of order $\beta\in(1,2]$ defined as \cite{jiang2012analytical} ($\frac{\partial^{\gamma}}{\partial\mid y\mid^{\gamma}}$ is defined similarly)
\begin{equation*}
\frac{\partial^{\beta}u(x,t)}{\partial|x|^{\beta}}=
\begin{cases}
-\frac{1}{2\cos(\beta\pi/2)\Gamma(2-\beta)}\frac{d^2}{dx^2}\int_{x_L}^{x_R}| x-\xi|^{1-\beta}u(\xi,t)d\xi,\quad 1<\beta<2,\\
\frac{\partial^{2}u(x,t)}{\partial
x^{2}},\quad \beta=2.\\
\end{cases}
\end{equation*}
Moreover, the time distributed-order operator ${}D_{t}^{\omega(\alpha)}$ is defined by \cite{meerschaert2011distributed}
\begin{align*}
 D_t^{\omega(\alpha)}u(\textbf{x},t)=\int_0^1\omega(\alpha){}^C_0D_t^{\alpha}u(\textbf{x},t)d\alpha,
\end{align*}
where $^C_0D_t^{\alpha}$ denotes the Caputo fractional derivative \cite{li2017analyticity} which is defined as follows:
\begin{equation*}
^C_0D_t^{\alpha}u(\textbf{x},t)=
\begin{cases}
\frac{1}{\Gamma(1-\alpha)}\int_0^t(t-\xi)^{-\alpha}\frac{\partial{u}}{\partial\xi}(\textbf{x},\xi)d\xi,\quad 0<\alpha<1,\\
u_t(\textbf{x},t),\quad \alpha=1,\\
\end{cases}
\end{equation*}
and the non-negative weight function $\omega(\alpha)$ satisfies that
\begin{equation*}
0\leq\omega(\alpha),~\omega(\alpha)\neq0,~\alpha\in[0,1],~0<\int_0^{1}\omega(\alpha)d\alpha<\infty,
\end{equation*}
where $\Gamma(\cdot)$ denotes the Gamma function.

Nonlocal behavior has been remarked as one of the main characteristics of the fractional differential operator.
As a result, most numerical methods for FDEs produce dense matrices or even full coefficient matrices in
one-dimensional cases \cite{lei2013circulant,gu2016fast}. Traditional methods, such as Gaussian elimination,
need computational workload of $\mathcal{O}(M^3)$ and memory capacity of $\mathcal{O}(M^2)$, where $M$ is the number of grid points \cite{lei2013circulant}. The Krylov subspace methods are studied and adopted to reduce the costs \cite{gu2016fast,gu2015hybridized,gu2015k,wang2011fast}. The convergent speed of the Krylov subspace
methods is dependent on the conditions of the discretized systems. To improve the performance of iterative
methods, many preconditioners \cite{lei2013circulant,Zhao2012DCT} are always designed according to the
structure of the linear systems. For one-dimensional cases, Wang et al. \cite{wang2011fast} made the important
discovery that the resultant systems had Toeplitz coefficient matrices. By exploiting this structure, the memory
requirement can be reduced from $\mathcal{O}(M^2)$ to $\mathcal{O}(M)$, and the fast Fourier transform (FFT)
can be used to evaluate the matrix-vector product in $\mathcal{O}(M\log M)$ operations. Moreover, the coefficient
matrices discretized from (\ref{1.1})-(\ref{1.3}) should be symmetric positive definite Toeplitz matrices due
to the existence of Riesz fractional derivatives \cite{Yang2010,ye2014numerical}. The circulant preconditioners
\cite{chan2007introduction,chan1989toeplitz,lei2013circulant} proved to be good choices to accelerate the
convergence of Krylov subspace methods when solving the discretized linear systems.
%
%
In high-dimensional cases, a nonsingular multilevel circulant preconditioner was proposed by Lei et al. \cite{leichenzhang2016}
to accelerate the convergence of Krylov subspace methods efficiently. In \cite{chou2017385}, Chou et al. illustrated
the efficiency of applying an approximate inverse preconditioner to the high dimensional FDEs when Krylov
subspace methods are employed. They also showed that under certain conditions, the normalized preconditioned matrix is equal to the sum of an identity matrix, a matrix with small norm, and a matrix with low rank, such that the preconditioned Krylov
subspace method converges superlinearly.

In this paper, we focus on establishing a fast numerical method and investigating the unconditional stability and
convergence for solving the TDRFDEs \eqref{1.1}-\eqref{1.3}. We first transform the TDRFDEs \eqref{1.1}-\eqref{1.3}
into the multi-term time-space FDEs based on the composite trapezoid formula. Then we apply the interpolation
approximation, as introduced by Gao et al. in \cite{gaotemporal}, to approximate the time derivatives of the multi-term
time-space FDEs at a special point. The global second-order numerical accuracy in time is independent
to the order of fractional derivatives. To gather numerical solutions with high-order accuracy in space, the fractional
centred difference formula \cite{ye2014numerical} is used to discrete the space Riesz derivative. Therefore we develop
a new difference scheme which converges with the second-order accuracy in time, space and distributed-order. On the other hand, by taking advantage of Toeplitz structure of the resultant linear systems, we adopt
the Krylov subspace method with efficient circulant preconditioners. It also proves that the eigenvalues
of the preconditioned matrices are clustered around 1, and the convergence rate of our proposed iterative method
is superlinear.

The rest of the paper is arranged as follows. In Section \ref{section2}, we study the TDRFDEs in one-dimensional case
and present its corresponding difference scheme. The uniqueness, unconditional stability and convergence of the difference
method are proved. Meanwhile, we design a preconditioned Krylov subspace method to solve the resultant Toeplitz linear system.
In Section \ref{section3}, the two-dimensional TDRFDE is discussed. We demonstrate that the difference scheme is uniquely
solvable, unconditionally stable and convergent with the convergence order $\mathcal{O}(h_1^2+h_2^2+\tau^2+\Delta\alpha^2)$.
We also adopt the preconditioned Krylov subspace method with suitable circulant preconditioners to handle the resulting
systems. Numerical experiments are carried out in Section \ref{section4} to illustrate the efficiency of our numerical
approaches. Finally, the paper closes with conclusions and remarks in Section \ref{section5}.

\section{One-dimensional problem}\label{section2}
\quad
Consider the following one-dimensional TDRFDE:
\begin{align}
&{}D_{t}^{\omega(\alpha)}u(x,t)=K\frac{\partial^{\beta}u(x,t)}{\partial|
x|^{\beta}}+f(x,t),\quad 0<x<L,~0<t\leq T,&\label{2.1}\\
&u(0,t)=0,\quad u(L,t)=0,\quad 0\leq t\leq T,&\label{2.2}\\
&u(x,0)=\phi(x),\quad 0<x<L.&\label{2.3}
\end{align}

In this section, we show that the discretizations for the distributed-order integral term of \eqref{2.1}-\eqref{2.3}
by the composite trapezoid formula lead to multi-term time-space FDE. We propose the second-order difference
scheme based on the interpolation approximation on a special point to solve the equations. We also
prove that the difference scheme is uniquely solvable, unconditionally stable and convergent with second-order accuracy
in time, space and distributed-order integral variables. Moreover, we propose an efficient implementation
based on Krylov subspace solver with suitable circulant preconditioners to solve the resultant Toeplitz linear system.

\subsection{Numerical discretization of the \eqref{2.1}-\eqref{2.3}}\label{section2.1}
\quad We first discretize the integral interval $[0,1]$ by the grid $0=\alpha_0<\alpha_1<\cdots<\alpha_{2J}=1$ with $\Delta\alpha=\frac{1}{2J}$ and $\alpha_l=l\Delta\alpha,~l=0,1,2,\cdots,2J$. The following lemma gives a complete description of the numerical approximation to the distributed-order integral term.

\begin{lemma}(\hspace*{-0.35em} The composite trapezoid formula \cite{gao2015two,gao2015some})
Let $z(\alpha)\in C^2([0,1])$, then we have
\begin{equation*}
\int_0^1z(\alpha)d\alpha=\Delta\alpha\sum\limits_{l=0}^{2J} d_lz(\alpha_l)-\frac{\Delta\alpha^2}{12}z^{(2)}(\eta),\quad \eta\in(0,1),
\end{equation*}
where
\begin{equation*}
d_l=
\begin{cases}
\frac{1}{2},\quad l=0,~2J,\\
1,\quad 1\leq l\leq2J-1.
\end{cases}
\end{equation*}
\label{lemma2.1}
\end{lemma}

 Considering the left side of \eqref{2.1}, let
$z(\alpha)=\omega(\alpha){}^C_0D_t^{\alpha}u(x,t)$ and using Lemma \ref{lemma2.1}, we can obtain
\begin{equation}
{}D_{t}^{\omega(\alpha)}u(x,t)=\Delta\alpha\sum\limits_{r=0}^{2J}d_r\omega(\alpha_r) {}^C_0D_t^{\alpha_r}u(x,t)+\mathcal{O}(\Delta\alpha^2).\label{2.4}
\end{equation}

Let $m=2J,~\lambda_r=d_r\omega(\alpha_r)\Delta\alpha$. The problem \eqref{2.1}-\eqref{2.3} is now converted into the following multi-term time-space FDE:
\begin{align}
&\sum\limits_{r=0}^{m}\lambda_r{}^C_0D_t^{\alpha_r}u(x,t)=K\frac{\partial^{\beta}u(x,t)}{\partial|
x|^{\beta}}+f(x,t),\quad 0<x<L,~0<t\leq T,&\label{2.5}\\
&u(0,t)=0,\quad u(L,t)=0,\quad 0\leq t\leq T,&\label{2.6}\\
&u(x,0)=\phi(x),\quad 0<x<L.&\label{2.7}
\end{align}

Next, we discrete the domain $[0,L]\times[0,T]$ with $ x_i=ih~(0\leq i\leq M)$ and $t_n=n\tau~(0\leq n\leq N)$, where $h=\frac{L}{M}$ and $\tau=\frac{T}{N}$ are space and time step sizes respectively.
Then we introduce the following preliminary lemma:
\begin{lemma}\hspace*{-0.35em} Suppose
\begin{align*}
F(\sigma)=\sum\limits_{r=0}^{m}\frac{\lambda_r}{\Gamma(3-\alpha_r)}\sigma^{1-\alpha_r}
\left[\sigma-\left(1-\frac{\alpha_r}{2}\right)\right]\tau^{2-\alpha_r},\quad \sigma\geq0.
\end{align*}
Let $a=\min\limits_{0\leq r\leq m}\left\{1-\frac{\alpha_r}{2}\right\},\quad b=\max\limits_{0\leq r\leq m}\left\{1-\frac{\alpha_r}{2}\right\},$
we can obtain that the equation $F(\sigma)=0$ has a unique positive root $\sigma^*\in[a,b]$,\\
where
\begin{align*}
a=1-\frac{1}{2}\max\limits_{0\leq r\leq m}\{\alpha_r\}=1-\frac{\alpha_m}{2}=\frac{1}{2},\quad b=1-\frac{1}{2}\min\limits_{0\leq r\leq m}\{\alpha_r\}=1-\frac{\alpha_0}{2}=1.
\end{align*}
\label{lemma2.2}
\end{lemma}

\begin{proof}
The proof is quite similar to Lemma 2.1 in \cite{gaotemporal} and thererfore is omitted.
\end{proof}

 For convenience, we let $\sigma=\sigma^*$, which means that $\sigma\in[\frac{1}{2},1]$ satisfies $F(\sigma)=0$.

Let $t_{n-1+\sigma}=(n-1+\sigma)\tau$, 
two lemmas are given below that will be useful in the discretizations of the multi-term time-space FDE later.
\begin{lemma}\hspace*{-0.35em} Suppose $u(t)\in C^3([t_0,t_n]) $, consider the linear combination of multi-term fractional derivatives $\sum\limits_{r=0}^{m}\lambda_r{}^C_0D_t^{\alpha_r}u(t)$ at the point $t=t_{n-1+\sigma}$, where $\lambda_r~(r=0,1,2,\cdots,m)>0$, $0\leq\alpha_0<\alpha_{1}<\cdots<\alpha_m\leq1$ and at least one of $\alpha_i$'s belongs to $(0,1)$.
The second-order accurate interpolation approximation for the $\sum\limits_{r=0}^{m}\lambda_r{}^C_0D_t^{\alpha_r}u(t)$ is as follows:
\begin{align*}
\sum\limits_{r=0}^{m}\lambda_r{}^C_0D_t^{\alpha_r}u(t_{n-1+\sigma})= \sum\limits_{k=0}^{n-1}\hat{c}_k^{(n)}\left[u(t_{n-k})-u(t_{n-k-1})\right]+\mathcal{O}(\tau^{3-\alpha_m}),
\end{align*}
where $$\hat{c}_k^{(n)}=\sum\limits_{r=0}^{m}\lambda_r\frac{\tau^{-\alpha_r}}{\Gamma(2-\alpha_r)}
c_k^{(n,\alpha_r)},$$
in which $c_0^{(n,\alpha_r)}=a_0^{(\alpha_r)}$, when $n=1$;\\
For $n\geq2$, we have
\begin{equation*}
c_k^{(n,\alpha_r)}=
\begin{cases}
a_0^{(\alpha_r)}+b_1^{(\alpha_r)},\quad\quad\quad\quad k=0,\\
a_k^{(\alpha_r)}+b_{k+1}^{(\alpha_r)}-b_k^{(\alpha_r)},\quad 1\leq k\leq n-2,\\
a_k^{(\alpha_r)}-b_k^{(\alpha_r)},\quad\quad\quad\quad k=n-1,
\end{cases}
\end{equation*}
where
\begin{align*}
&a_0^{\alpha_r}=\sigma^{1-\alpha_r};~ a_l^{\alpha_r}=(l+\sigma)^{1-\alpha_r}-(l-1+\sigma)^{1-\alpha_r}~(l\geq1),\\
&b_l^{\alpha_r}=\frac{1}{2-\alpha_r}\left[(l+\sigma)^{2-\alpha_r}-(l-1+\sigma)^{2-\alpha_r}\right]
-\frac{1}{2}\left[(l+\sigma)^{1-\alpha_r}+(l-1+\sigma)^{1-\alpha_r}\right].
\end{align*}
In particular, when $\alpha_r=1$, we have $c_0^{(n,\alpha_r)}=1,~c_k^{(n,\alpha_r)}=0~(1\leq k\leq n-1)$; when $\alpha_r=0$, we have $c_0^{(n,\alpha_r)}=\sigma,~c_k^{(n,\alpha_r)}=1~(1\leq k\leq n-1)$.
\label{lemma2.3}
\end{lemma}

\begin{proof}
For a rigorous proof of this lemma, the reader is referred to \cite{gaotemporal}.
\end{proof}

\begin{lemma} \cite{ye2014numerical} Suppose that $u(x)\in C^5[0,L]$ satisfy the boundary condition $u(0)=u(L)=0$. The fractional centred difference formula for approximating the Riesz derivatives when $1<\beta\leq2$ is as follows:
\begin{equation*}
\frac{\partial^{\beta}u(x_i)}{\partial| x|^{\beta}}=
-h^{-\beta}\sum\limits_{k=i-M}^ig_k^{(\beta)}u(x_{i-k})+\mathcal{O}(h^2),
\end{equation*}
where
\begin{equation*}
g_k^{(\beta)}=\frac{(-1)^k\Gamma(\beta+1)}{\Gamma(\beta/2-k+1)\Gamma(\beta/2+k+1)}.
\end{equation*}
\label{lemma2.4}
\end{lemma}

Assume that $u(x,t)\in C^{(5,3)}([0,L]\times[0,T])$ is a solution to the problem \eqref{2.1}-\eqref{2.3}. According to equation \eqref{2.5} at $(x_i,t_{n-1+\sigma})$, we get
\begin{align}
\nonumber
&\sum\limits_{r=0}^{m}\lambda_r{}^C_0D_t^{\alpha_r}u(x_i,t_{n-1+\sigma})\\
=&K\frac{\partial^{\beta}u(x_i,t_{n-1+\sigma})}{\partial|
x|^{\beta}}+f(x_i,t_{n-1+\sigma}),\quad 1\leq i\leq M-1,~1\leq n\leq N.\label{2.8}
\end{align}
For simplicity, we define
\begin{align*}
&U_i^n=u(x_i,t_n),\quad0\leq i\leq M,~0\leq n\leq N;\\
&f_i^{n-1+\sigma}=f(x_i,t_{n-1+\sigma}),\quad0\leq i\leq M,~1\leq n\leq N.
\end{align*}
Using Lemma \ref{lemma2.3}, we have
\begin{align}
\sum\limits_{r=0}^{m}\lambda_r{}^C_0D_t^{\alpha_r}u(x_i,t_{n-1+\sigma})
=\sum\limits_{k=0}^{n-1}\hat{c}_k^{(n)}\left(U_i^{n-k}-U_i^{n-k-1}\right)+\mathcal{O}(\tau^{3-\alpha_m}).\label{2.9}
\end{align}
By applying the second-order linear interpolation formula to the Riesz derivative on the right side of equation \eqref{2.8}, we obtain that
\begin{equation}
\frac{\partial^{\beta}u(x_i,t_{n-1+\sigma})}{\partial| x|^{\beta}}=
\sigma\frac{\partial^{\beta}u(x_i,t_n)}{\partial|
x|^{\beta}}+(1-\sigma)\frac{\partial^{\beta}u(x_i,t_{n-1})}{\partial|
x|^{\beta}}+\mathcal{O}(\tau^{2}).\label{2.10}
\end{equation}
Furthermore, based on Lemma \ref{lemma2.4}, we have
\begin{equation}
\frac{\partial^{\beta}u(x_i,t_n)}{\partial| x|^{\beta}}=
-h^{-\beta}\sum\limits_{k=i-M}^ig_k^{(\beta)}U_{i-k}^n+\mathcal{O}(h^2).\label{2.11}
\end{equation}
Combine formulae \eqref{2.10} and \eqref{2.11}, and we get
\begin{equation}
\frac{\partial^{\beta}u(x_i,t_{n-1+\sigma})}{\partial| x|^{\beta}}=
-h^{-\beta}\sum\limits_{k=i-M}^ig_k^{(\beta)}\left[\sigma U_{i-k}^n+(1-\sigma)U_{i-k}^{n-1}\right]+\mathcal{O}(h^2+\tau^2).\label{2.12}
\end{equation}
By substituting \eqref{2.9} and \eqref{2.12} into \eqref{2.8}, we obtain
\begin{align}\label{2.13}
\nonumber
\sum\limits_{k=0}^{n-1}\hat{c}_k^{(n)}\left(U_i^{n-k}-U_i^{n-k-1}\right)
=&-Kh^{-\beta}\sum\limits_{k=i-M}^ig_k^{(\beta)}\left[\sigma U_{i-k}^n+(1-\sigma)U_{i-k}^{n-1}\right]\\
&+f_i^{n-1+\sigma}+R_i^n,1\leq i\leq M-1,~1\leq n\leq N,
\end{align}
where there exists a positive constant $c_1$ such that
\begin{align}
\mid R_i^n\mid\leq c_1\left(h^2+\tau^2+\Delta\alpha^2\right),\quad 1\leq i\leq M-1,~1\leq n\leq N.\label{2.14}
\end{align}

Notice the initial-boundary conditions \eqref{2.6}-\eqref{2.7}. We have
\begin{align}
&U_0^n=0,\quad U_M^n=0,\quad 0\leq n\leq N,\label{2.15}\\
&U_i^0=\phi(x_i),\quad 1\leq i\leq M-1.\label{2.16}
\end{align}

Suppose $u_i^k$ is the numerical approximation to $u(x_i,t_k)$. By omitting the local truncation error term $R_i^n$ in \eqref{2.13} and replacing the exact solution $U_i^n$ with $u_i^k$ in \eqref{2.13}, \eqref{2.15}-\eqref{2.16}, we can construct the following difference scheme for the \eqref{2.1}-\eqref{2.3}:
\begin{align}
\nonumber
&\sum\limits_{k=0}^{n-1}\hat{c}_k^{(n)}\left(u_i^{n-k}-u_i^{n-k-1}\right)
=-Kh^{-\beta}\sum\limits_{k=i-M}^ig_k^{(\beta)}\left[\sigma u_{i-k}^n+(1-\sigma)u_{i-k}^{n-1}\right]+f_i^{n-1+\sigma},\\
&\qquad\qquad\qquad\qquad\qquad\qquad 1\leq i\leq M-1,~1\leq n\leq N,\label{2.17}\\
&u_0^n=0,\quad u_M^n=0,\quad 0\leq n\leq N,\label{2.18}\\
&u_i^0=\phi(x_i),\quad 1\leq i\leq M-1.\label{2.19}
\end{align}

\subsection{Solvability, stability and convergence analysis}\label{section2.2}
 \quad In this subsection, we analyze the unique solvability, unconditional stability and convergence of the difference scheme \eqref{2.17}-\eqref{2.19} obtained in Section \ref{section2.1}. Meanwhile, we show that the convergence accuracy for the proposed difference scheme is second-order in space, in time and in distributed-order integral in the mesh $L_2$-norm.

We define
\begin{align*}
V_h=\{v~|~v=(v_0,v_1,\cdots,v_{M-1},v_M)^T,~v_0=0,~v_M=0\}.
\end{align*}
For all $v,~w\in V_h$, the discrete inner product and the corresponding discrete $L_2$-norm are defined as follows:
\begin{align*}
(v,w)=h\sum\limits_{i=1}^{M-1}v_iw_i,\quad and \quad \|v\|=\sqrt{(v,v)}.
\end{align*}

Before introducing the properties on the solvability, unconditional stability and convergence, several useful lemmas are prepared below.

\begin{lemma} \cite{ye2014numerical} Let $1<\beta\leq2$ and take $g_k^{(\beta)}$ as defined in Lemma \ref{lemma2.4}. We have
\begin{equation*}
\begin{cases}
g_0^{(\beta)}=\frac{\Gamma(\beta+1)}{\Gamma^2(\beta/2+1)}\geq0,\quad g_{-k}^{(\beta)}=g_k^{(\beta)}\leq0,\quad k=1,2,\cdots,\\
\sum\limits_{k=-\infty}^{\infty}g_k^{(\beta)}=0,\quad -\sum\limits_{k=-M+i\atop k\neq0}^{i}g_k^{(\beta)}\leq g_0^{(\beta)},\quad 1\leq i\leq M-1,\\
 g_k^{(\beta)}=\left(1-\frac{\beta+1}{\beta/2+k}\right)g_{k-1}^{(\beta)},\quad k\geq1.
\end{cases}
\end{equation*}
\label{lemma2.5}
\end{lemma}

\begin{lemma} \cite{gaotemporal} Let
$\hat{c}_k^{(n)}=\sum\limits_{r=0}^{m}\lambda_r\frac{\tau^{-\alpha_r}}{\Gamma(2-\alpha_r)}
c_k^{(n,\alpha_r)},~k=0,1,\cdots,n-1,$ as is defined in Lemma \ref{lemma2.3}, it holds
\begin{equation*}
\hat{c}_0^{(n)}>\hat{c}_1^{(n)}>\cdots>\hat{c}_{n-2}^{(n)}>\hat{c}_{n-1}^{(n)}>
\sum\limits_{r=0}^{m}\lambda_r\frac{\tau^{-\alpha_r}}{\Gamma(2-\alpha_r)}\cdot\frac{1-\alpha_r}{2}
(n-1+\sigma)^{-\alpha_r}.
\end{equation*}
\label{lemma2.6}
\end{lemma}

\begin{lemma} \cite{alikhanov2015new} Let $V$ represent the inner product space and $(\cdot,\cdot)$ denote the inner product with the induced norm $\|\cdot\|$. For $v^0,~ v^1,\cdots,~v^n\in V$, when $n\geq1$ we have
\begin{equation*}
\sum\limits_{k=0}^{n-1}\hat{c}_k^{(n)}\left(v^{n-k}-v^{n-k-1},~\sigma v^n+(1-\sigma)v^{n-1}\right)
\geq\frac{1}{2}\sum\limits_{k=0}^{n-1}\hat{c}_k^{(n)}\left(\|v^{n-k}\|^2-\|v^{n-k-1}\|^2\right).
\end{equation*}
\label{lemma2.7}
\end{lemma}

\begin{lemma} \cite{sun2015finite} For $1<\beta\leq2$ and any $v\in V_h$, it holds that
\begin{equation*}
-h^{-\beta}h\sum\limits_{i=1}^{M-1}\left(\sum\limits_{k=i-M}^ig_k^{(\beta)}v_{i-k}\right)v_i
\leq-c_*^{(\beta)}(2L)^{-\beta}h\sum\limits_{i=1}^{M-1}v_i^2,
\end{equation*}
where $c_*^{(\beta)}=\frac{2}{\beta}r_\beta,$
with
$$r_\beta=e^{-2}\frac{(4-\beta)(2-\beta)\beta}{(6+\beta)(4+\beta)(2+\beta)}\cdot
\frac{\Gamma(\beta+1)}{\Gamma^2(\beta/2+1)}\left(3+\frac{\beta}{2}\right)^{\beta+1}.$$
\label{lemma2.8}
\end{lemma}

First, let us consider the unique solvability of the proposed numerical method \eqref{2.17}-\eqref{2.19}.

\begin{theorem}
The difference scheme \eqref{2.17}-\eqref{2.19} is uniquely solvable.
\label{th2.1}
\end{theorem}
\begin{proof}
Let $u^n=(u_0^n,u_1^n,u_2^n,\cdots,u_{M-1}^n,u_{M}^n)^T.$ According to \eqref{2.18} and \eqref{2.19}, the value of $u^0$ is determined. Now suppose that $\{u^k~|~0\leq k\leq n-1\}$ has been determined. According to \eqref{2.17} and \eqref{2.18}, we get a linear equation system with respect to $u^n$. Then we only need to prove that the corresponding homogeneous linear system
\begin{align}
&\hat{c}_0^{(n)}u_i^n
=-K\sigma h^{-\beta}\sum\limits_{k=i-M}^ig_k^{(\beta)}u_{i-k}^n,\quad 1\leq i\leq M-1,\label{2.20}\\
&u_0^n=0,\quad u_M^n=0\label{2.21}
\end{align}
only has solution of 0.

We first rewrite the equation \eqref{2.20} as follows:
\begin{align}
\left[\hat{c}_0^{(n)}+K\sigma h^{-\beta} g_0^{(\beta)}\right]u_i^n
=K\sigma h^{-\beta}\sum\limits_{k=i-M\atop k\neq0}^i\left(-g_k^{(\beta)}\right)u_{i-k}^n,\quad 1\leq i\leq M-1.\label{2.22}
\end{align}
Let $\|u^n\|_\infty=\mid u_{i_n}^n\mid$, where $i_n\in\{1,2,\cdots,M-1\}$. Let us consider equation \eqref{2.22} with $i=i_n$ and take absolute values on both sides of the equation. Based on Lemma \ref{lemma2.5} and the fact that the coefficients $K>0$, it can be seen that
\begin{align*}
&\left[\hat{c}_0^{(n)}+K\sigma h^{-\beta} g_0^{(\beta)}\right]\parallel u^n\parallel_\infty\\
\leq&K\sigma h^{-\beta}\sum\limits_{k=i_n-M\atop k\neq0}^{i_n}\left(-g_k^{(\beta)}\right)\mid u_{i_n-k}^n\mid\\
\leq&K\sigma h^{-\beta}\sum\limits_{k=i_n-M\atop k\neq0}^{i_n}\left(-g_k^{(\beta)}\right)\parallel u^n\parallel_\infty\\
\leq&K\sigma h^{-\beta} g_0^{(\beta)}\parallel u^n\parallel_\infty.
\end{align*}

Therefore, $\parallel u^n\parallel_\infty=0$ is derived, which indicates that the homogeneous linear equations \eqref{2.20}-\eqref{2.21}  have a single solution of 0. 
\end{proof}

We are now going to prove the unconditional stability of the difference scheme \eqref{2.17}-\eqref{2.19} with respect to the initial value and the inhomogeneous term $f(x,t)$. The correlation result is shown in the following theorem.

\begin{theorem}
 Let $\{u_i^n~|~0\leq i\leq M,~0\leq n\leq N\}$ be the solution of the difference scheme \eqref{2.17}-\eqref{2.19}. We have
\begin{align*}
\parallel u^n\parallel^2\leq\parallel u^0\parallel^2+\frac{(2L)^\beta}{Kc_*^{(\beta)} \sum\limits_{r=0}^{m}\frac{\lambda_r}{T^{\alpha_r}{\Gamma(1-\alpha_r)}}}\max\limits_{1\leq l\leq n}\parallel f^{l-1+\sigma}\parallel^2,\quad 1\leq n\leq N,
\end{align*}
where
$$\parallel f^{l-1+\sigma}\parallel^2=h\sum\limits_{i=1}^{M-1}\left(f_i^{l-1+\sigma}\right)^2.$$
\label{th2.2}
\end{theorem}
\begin{proof}
Multiplying \eqref{2.17} by $h(\sigma u_i^n+(1-\sigma)u_i^{n-1})$ and summing up with $i$ from $1$ to $M-1$, we get
\begin{align}
\nonumber
&\sum\limits_{k=0}^{n-1}\hat{c}_k^{(n)}h\sum\limits_{i=1}^{M-1}\left(u_i^{n-k}-u_i^{n-k-1}\right)\left[\sigma u_i^n+(1-\sigma)u_i^{n-1}\right]\label{2.23}\\
\nonumber
=&-Kh^{-\beta}h\sum\limits_{i=1}^{M-1}\sum\limits_{k=i-M}^ig_k^{(\beta)}\left[\sigma u_{i-k}^n+(1-\sigma)u_{i-k}^{n-1}\right]\left[\sigma u_i^n+(1-\sigma)u_i^{n-1}\right]\\
&+h\sum\limits_{i=1}^{M-1}f_i^{n-1+\sigma}\left[\sigma u_i^n+(1-\sigma)u_i^{n-1}\right],\quad 1\leq n\leq N.
\end{align}
According to Lemma \ref{lemma2.7}, it follows that
\begin{align}\label{2.24}
\nonumber
&\sum\limits_{k=0}^{n-1}\hat{c}_k^{(n)}h\sum\limits_{i=1}^{M-1}\left(u_i^{n-k}-u_i^{n-k-1}\right)\left[\sigma u_i^n+(1-\sigma)u_i^{n-1}\right]\\
\nonumber
=&\sum\limits_{k=0}^{n-1}\hat{c}_k^{(n)}\left(u^{n-k}-u^{n-k-1},\sigma u^n+(1-\sigma)u^{n-1}\right)\\
\geq&\frac{1}{2}\sum\limits_{k=0}^{n-1}\hat{c}_k^{(n)}\left(\parallel u^{n-k}\parallel^2-\parallel u^{n-k-1}\parallel^2\right).
\end{align}
Using Lemma \ref{lemma2.8}, we obtain
\begin{align}
\nonumber
&-Kh^{-\beta} h\sum\limits_{i=1}^{M-1}\sum\limits_{k=i-M}^ig_k^{(\beta)}\left[\sigma u_{i-k}^n+(1-\sigma)u_{i-k}^{n-1}\right]\left[\sigma u_i^n+(1-\sigma)u_i^{n-1}\right]\label{2.25}\\
\leq&-Kc_*^{(\beta)}(2L)^{-\beta}\parallel \sigma u^n+(1-\sigma)u^{n-1}\parallel^2.
\end{align}
In addition, by exploiting Cauchy-Schwarz inequality, we can get
\begin{align}
\nonumber
&h\sum\limits_{i=1}^{M-1}f_i^{n-1+\sigma}\left[\sigma u_i^n+(1-\sigma)u_i^{n-1}\right]\label{2.26}\\
\nonumber
\leq&\parallel f^{n-1+\sigma}\parallel\cdot\parallel \sigma u^n+(1-\sigma)u^{n-1}\parallel\\
\leq&Kc_*^{(\beta)}(2L)^{-\beta}\parallel \sigma u^n+(1-\sigma)u^{n-1}\parallel^2+
\frac{(2L)^{\beta}}{4Kc_*^{(\beta)}}\parallel f^{n-1+\sigma}\parallel^2.
\end{align}
By substituting \eqref{2.24}-\eqref{2.26} into \eqref{2.23}, we have
\begin{align}
\frac{1}{2}\sum\limits_{k=0}^{n-1}\hat{c}_k^{(n)}\left(\parallel u^{n-k}\parallel^2-\parallel u^{n-k-1}\parallel^2\right)\leq\frac{(2L)^{\beta}}{4Kc_*^{(\beta)}}\parallel f^{n-1+\sigma}\parallel^2,\quad 1\leq n\leq N.\label{2.27}
\end{align}

With the use of Lemma \ref{lemma2.6}, we get
\begin{align}
\hat{c}_{n-1}^{(n)}\geq
\sum\limits_{r=0}^{m}\lambda_r\frac{\tau^{-\alpha_r}}{\Gamma(2-\alpha_r)}\cdot\frac{1-\alpha_r}{2}
(n-1+\sigma)^{-\alpha_r}\geq \frac{1}{2}\sum\limits_{r=0}^{m}\frac{\lambda_r}{T^{\alpha_r}{\Gamma(1-\alpha_r)}}.\label{2.28}
\end{align}
Combine \eqref{2.27} and \eqref{2.28}, and we arrives at the following inequality:
\begin{align*}
\nonumber
\hat{c}_0^{(n)}\parallel u^{n}\parallel^2
\leq&\sum\limits_{k=1}^{n-1}\left(\hat{c}_{k-1}^{(n)}-\hat{c}_k^{(n)}\right)\parallel u^{n-k}\parallel^2+\hat{c}_{n-1}^{(n)}\parallel u^{0}\parallel^2
+\frac{(2L)^{\beta}}{2Kc_*^{(\beta)}}\parallel f^{n-1+\sigma}\parallel^2\\
\leq&\sum\limits_{k=1}^{n-1}\left(\hat{c}_{k-1}^{(n)}-\hat{c}_k^{(n)}\right)\parallel u^{n-k}\parallel^2\\
&+\hat{c}_{n-1}^{(n)}\left(\parallel u^{0}\parallel^2+\frac{(2L)^{\beta}}{Kc_*^{(\beta)}\sum\limits_{r=0}^{m}\frac{\lambda_r}
{T^{\alpha_r}{\Gamma(1-\alpha_r)}}}\parallel f^{n-1+\sigma}\parallel^2\right),~1\leq n\leq N.
\end{align*}
By applying the mathematical induction method to the above inequality, we can get
\begin{align*}
\parallel u^n\parallel^2\leq\parallel u^0\parallel^2+\frac{(2L)^\beta}{Kc_*^{(\beta)} \sum\limits_{r=0}^{m}\frac{\lambda_r}{T^{\alpha_r}{\Gamma(1-\alpha_r)}}}\max\limits_{1\leq l\leq n}\parallel f^{l-1+\sigma}\parallel^2,\quad 1\leq n\leq N.
\end{align*}
This completes the proof.
\end{proof}
We have established the unconditional stability of our difference scheme \eqref{2.17}-\eqref{2.19}, and now we further show its convergence.

Suppose that $\{U_i^n~|~0\leq i\leq M,~0\leq n\leq N\}$ is the exact solution of the system \eqref{2.1}-\eqref{2.3} and $\{u_i^n~|~0\leq i\leq M,~0\leq n\leq N\}$ is the numerical solution of the difference scheme \eqref{2.17}-\eqref{2.19}.
 Let $e_i^n=U_i^n-u_i^n~(0\leq i\leq M,~0\leq n\leq N)$.

By subtracting \eqref{2.17}-\eqref{2.19} from \eqref{2.13}, \eqref{2.15}-\eqref{2.16}, respectively, we obtain the system of error equations as follows:
\begin{align*}
\nonumber
&\sum\limits_{k=0}^{n-1}\hat{c}_k^{(n)}\left(e_i^{n-k}-e_i^{n-k-1}\right)
=-Kh^{-\beta}\sum\limits_{k=i-M}^ig_k^{(\beta)}\left[\sigma e_{i-k}^n+(1-\sigma)e_{i-k}^{n-1}\right]+R_i^n,\\
&\qquad\qquad\qquad\qquad\qquad\qquad 1\leq i\leq M-1,~1\leq n\leq N,\\
&e_0^n=0,\quad e_M^n=0,\quad 0\leq n\leq N,\\
&e_i^0=0,\quad 1\leq i\leq M-1.
\end{align*}
By applying the conclusion of Theorem \ref{th2.2} and noticing \eqref{2.14}, we have
\begin{align*}
\parallel e^n\parallel^2&\leq\frac{(2L)^{\beta}}{Kc_*^{(\beta)} \sum\limits_{r=0}^{m}\frac{\lambda_r}{T^{\alpha_r}{\Gamma(1-\alpha_r)}}}\max\limits_{1\leq l\leq n}\parallel R^l\parallel^2\\
&\leq\frac{(2L)^{\beta}}{Kc_*^{(\beta)} \sum\limits_{r=0}^{m}\frac{\lambda_r}{T^{\alpha_r}{\Gamma(1-\alpha_r)}}}  \left[c_1\left(h^2+\tau^2+\Delta\alpha^2\right)\right]^2L,\quad 1\leq n\leq N.
\end{align*}
Extract the square root on both sides of the equation above, then we acquire
\begin{align*}
\parallel e^n\parallel\leq c_1\sqrt{\frac{2^\beta L^{\beta+1}}{Kc_*^{(\beta)} \sum\limits_{r=0}^{m}\frac{\lambda_r}{T^{\alpha_r}{\Gamma(1-\alpha_r)}}}}   \left(h^2+\tau^2+\Delta\alpha^2\right),\quad 1\leq n\leq N.
\end{align*}

Therefore, we can get the following theorem.
\begin{theorem}
Suppose that the continuous problem \eqref{2.1}-\eqref{2.3} has a smooth solution $u(x,t)\in C^{(5,3)}(\Omega\times[0,T])$, and let $u_i^n$ be the solution of the difference scheme \eqref{2.17}-\eqref{2.19}. It holds that
\begin{align*}
\parallel e^n\parallel\leq c_1\sqrt{\frac{2^\beta L^{\beta+1}}{Kc_*^{(\beta)} \sum\limits_{r=0}^{m}\frac{\lambda_r}{T^{\alpha_r}{\Gamma(1-\alpha_r)}}}}   \left(h^2+\tau^2+\Delta\alpha^2\right),\quad 1\leq n\leq N.
\end{align*}
\label{th2.3}
\end{theorem}

\subsection{Fast solution techniques with circulant preconditioner}\label{section2.3}
\quad We rewrite the proposed implicit difference scheme \eqref{2.17} as the following matrix form at the time level $n$:
\begin{equation}\label{2.29}
 A^{n}u^{n}=b^{n-1},\quad n=1,2,\ldots,N,
\end{equation}
 where
 \begin{equation}\label{2.30}
 A^{n}=\hat{c}_0^{(n)}I+\sigma Kh^{-\beta}G,
 \end{equation}
 and
 \begin{equation*}
 b^{n-1}=-(1-\sigma)Kh^{-\beta}Gu^{n-1}+\sum\limits_{k=1}^{n-1}(\hat{c}_{k-1}^{(n)}-\hat{c}_k^{(n)})u^{n-k}+\hat{c}_{n-1}^{(n)}u^0+f^{n-1+\sigma}.
 \end{equation*}
Here $I$ is the identity matrix of order $M-1$ and
\begin{equation}\label{2.31}
G=
\begin{bmatrix}
{g}_0^{(\beta)}&{g}_{-1}^{(\beta)}&{g}_{-2}^{(\beta)}&\cdots&{g}_{3-M}^{(\beta)}&{g}_{2-M}^{(\beta)}\\
{g}_1^{(\beta)}&{g}_{0}^{(\beta)}&{g}_{-1}^{(\beta)}&\cdots&{g}_{4-M}^{(\beta)}&{g}_{3-M}^{(\beta)}\\
{g}_2^{(\beta)}&{g}_{1}^{(\beta)}&{g}_{0}^{(\beta)}&\cdots&{g}_{5-M}^{(\beta)}&{g}_{4-M}^{(\beta)}\\
\vdots&\vdots&\vdots&\ddots&\vdots&\vdots\\
{g}_{M-3}^{(\beta)}&{g}_{M-4}^{(\beta)}&{g}_{M-5}^{(\beta)}&\cdots&{g}_{0}^{(\beta)}&{g}_{-1}^{(\beta)}\\
{g}_{M-2}^{(\beta)}&{g}_{M-3}^{(\beta)}&{g}_{M-4}^{(\beta)}&\cdots&{g}_{1}^{(\beta)}&{g}_{0}^{(\beta)}
\end{bmatrix}.
\end{equation}
It is obvious that $G$ is a symmetric Toeplitz matrix (see \cite{lei2013circulant}). Therefore, it can be stored with only $M-1$ entries and the fast Fourier transform (FFT) can be used to carry out the matrix-vector product in only $\mathcal{O}((M-1)\log(M-1))$ operations.

The following lemma guarantees the invertibility of the matrix $A^{n}$ defined in \eqref{2.30}.

\begin{lemma}
 The coefficient matrix
\begin{equation*}
 A^{n}=\hat{c}_0^{(n)}I+\sigma Kh^{-\beta}G
 \end{equation*}
of the linear system \eqref{2.29} is a symmetric positive definite matrix.
\label{lemma2.9}
\end{lemma}

\begin{proof}
Let $a_{ij}^n$ be the $(i,j)$ entry of the $A^n$. We notice Lemma \ref{lemma2.5} and $\hat{c}_0^{(n)}>0$, thus
\begin{equation*}
\begin{split}
& |a_{ii}^n|-\sum\limits_{j=1,j\neq i}^{M-1}|a_{ij}^n| \\
 = &(\hat{c}_0^{(n)}+\sigma Kh^{-\beta}g_0^{(\beta)})-\sigma Kh^{-\beta}(\sum\limits_{j=-M+i+1,j\neq 0}^{i-1}|g_j^{(\beta)}|) \\
 =&\hat{c}_0^{(n)}+\sigma Kh^{-\beta}\sum\limits_{j=-M+i+1}^{i-1}g_j^{(\beta)}\\
>&\hat{c}_0^{(n)}>0.
\end{split}
\end{equation*}
This implies that $A^n$ is a strictly diagonally dominant matrix. According to Lemma \ref{lemma2.5}, it is easy to prove the symmetry of the coefficient matrix $A^n$ and all the main diagonal elements of $A^n$ are positive. Hence, all its eigenvalues are positive. So the coefficient matrix is a symmetric positive definite matrix.
\end{proof}
 It is well-known that the conjugate gradient (CG) method is a popular and effective Krylov subspace method \cite{lei2013circulant} for solving symmetric positive systems with Toeplitz coefficient matrix. Nevertheless, the drawback of the CG method is its slow convergence when the eigenvalues of the coefficient matrix $A^n$ are not clustered \cite{chan1989toeplitz}. To overtake this shortcoming, we use the CG method with a circulant preconditioner (PCG) to solve such linear systems \cite{lei2013circulant}.

We propose a circulant preconditioner, which is generated from the famous R. Chan circulant preconditioner \cite{chan2007introduction} to solve the Toeplitz linear system \eqref{2.29}. For a Toeplitz matrix $G_n\in\mathbb{C}^{n\times n}$ with form of \eqref{2.31}, the R. chan circulant preconditioner $R_n$  makes use of all the entries \cite{chan2007introduction}. Its entries $r_{ij}=r_{i-j}$ are given by
\begin{eqnarray*}
r_k=
\begin{cases}
g_0,    &k=0,\\
g_k+g_{k-n},      &0<k<n\\
r_{k+n},&0<-k<n.
\end{cases}
\end{eqnarray*}
Then the PCG method is employed to solve the following preconditioned system
\begin{equation*}
 (C^{n})^{-1}A^{n}u^{n}=(C^{n})^{-1}b^{n-1},\quad n=1,2,\ldots,N,
\end{equation*}
and the R. Chan-based circulant preconditioner $C^{n}$ takes the following form
\begin{equation*}
 C^{n}=\hat{c}_0^{(n)}I+\sigma Kh^{-\beta}c(G).
 \end{equation*}
 More precisely, the first column of $c(G)$ is given by
 \begin{center}
$\left(
  \begin{array}{c}
    g_0^{(\beta)} \\
   g_1^{(\beta)}+g_{2-M}^{(\beta)} \\
   g_2^{(\beta)}+g_{3-M}^{(\beta)} \\
    \vdots \\
    \vdots \\
    g_{M-3}^{(\beta)}+g_{-2}^{(\beta)} \\ \\
    g_{M-2}^{(\beta)}+g_{-1}^{(\beta)} \\ \\
  \end{array}
\right)$.
 \end{center}
 Below we discuss the basic properties of the circulant preconditioner $C^{n}$.

 \begin{lemma}
 The circulant preconditioner
\begin{equation*}
 C^{n}=\hat{c}_0^{(n)}I+\sigma Kh^{-\beta}c(G)
 \end{equation*}
is a symmetric positive definite matrix.
\label{lemma2.10}
\end{lemma}

\begin{proof}
As similar to Lemma \ref{lemma2.9}, suppose $c_{ij}^n$ be the $(i,j)$ entry of $C^n$. Based on Lemma \ref{lemma2.5} and $\hat{c}_0^{(n)}>0$ we get
\begin{equation*}
\begin{split}
& |c_{ii}^n|-\sum\limits_{j=1,j\neq i}^{M-1}|c_{ij}^n| \\
 =& (\hat{c}_0^{(n)}+\sigma Kh^{-\beta}g_0^{(\beta)})-\sigma Kh^{-\beta}(\sum\limits_{j=1}^{M-2}|g_j^{(\beta)}+g_{-j}^{(\beta)}|) \\
 =&\hat{c}_0^{(n)}+\sigma Kh^{-\beta}\sum\limits_{j=2-M}^{M-2}g_j^{(\beta)}\\
>&\hat{c}_0^{(n)}>0,
\end{split}
\end{equation*}
which implies that $C^n$ is a strictly diagonally dominant matrix. From Lemma \ref{lemma2.5}, we can easily know that the main diagonal elements of $C^n$ are positive and $C^n$ is symmetric. Therefore, $C^n$ is a symmetric positive definite matrix.
\end{proof}
 Lemma \ref{lemma2.10} suggests that the preconditioner $C^n$ is invertible. In addition, the eigenvalue distributions of preconditioned matrices $(C^{n})^{-1}A^{n}$ are theoretically proven to be clustered around 1 \cite{chan2007introduction}. The convergence rate of PCG is superlinear
\cite{chan1989toeplitz}. We will demonstrate numerically that the circulant preconditioning exhibits nice clustering eigenvalues in Section \ref{section4}. It is both numerically and theoretically guaranteed that the computational cost per iteration of PCG is $\mathcal{O}((M-1)\log(M-1))$ and the total cost at each time step is $\mathcal{O}((M-1)\log(M-1))$.

\section{Two-dimensional problem}\label{section3}

\quad Consider the following two-dimensional TDRFDE:
\begin{align}
\nonumber
&{}D_{t}^{\omega(\alpha)}u(x,y,t)=K_1\frac{\partial^{\beta}u(x,y,t)}{\partial|
x|^{\beta}}+K_2\frac{\partial^{\gamma}u(x,y,t)}{\partial|
y|^{\gamma}}+f(x,y,t),\\
&\qquad\qquad\qquad\qquad\qquad\qquad\qquad\qquad (x,y)\in\Omega,~0<t\leq T,\label{3.1}\\
&u(x,y,t)=0,\quad (x,y)\in\partial\Omega,~0\leq t\leq T,\label{3.2}\\
&u(x,y,0)=\phi(x,y),\quad (x,y)\in\Omega,\label{3.3}
\end{align}
where $\Omega=(0,L_1)\times(0,L_2)$, $\partial\Omega$ is the boundary of $\Omega$, $f(x,y,t)$ and $\phi(x,y)$ are given functions. Especially, $\phi(x,y)=0$ holds when $(x,y)\in\partial\Omega$.

In this section, we can directly extend the idea for solving the one-dimensional problem \eqref{2.1}-\eqref{2.3} to handle the two-dimensional problem \eqref{3.1}-\eqref{3.3}. We propose a second-order difference scheme based on the interpolation approximation on a special point to solve the two-dimensional TDRFDE. The unique solvability, unconditional stability and convergence of the proposed difference scheme are also discussed. Furthermore, a multilevel circulant preconditioner is proposed to accelerate the convergence rate of the Krylov subspace method.

\subsection{Numerical discretization for \eqref{3.1}-\eqref{3.3}}\label{section3.1}
\quad To derive the difference scheme of \eqref{3.1}-\eqref{3.3}, we first divide the interval $[0,L_1]$ into $M_1$-subintervals with $h_1=\frac{L_1}{M_1}$ and $x_i=ih_1~(0\leq i\leq M_1)$, and divide the interval $[0,L_2]$ into $M_2$-subintervals with $h_2=\frac{L_2}{M_2}$ and $y_j=jh_2~(0\leq j\leq M_2)$.

Denote $\omega=\{(i,j)~|~1\leq i\leq M_1-1,~1\leq j\leq M_2-1\}$, $\partial\omega=\{(i,j)~|~(x_i,y_j)\in\partial\Omega\}$, $\bar{\omega}=\omega\bigcup\partial\omega$.\\
For simplicity, we define
\begin{align*}
&U_{ij}^n=u(x_i,y_j,t_n),\quad(i,j)\in\bar{\omega},~0\leq n\leq N;\\
&f_{ij}^{n-1+\sigma}=f(x_i,y_j,t_{n-1+\sigma}),\quad(i,j)\in\bar{\omega},~1\leq n\leq N.
\end{align*}
Suppose $u(x,y,t)\in C^{(5,5,3)}~(\Omega\times[0,T])$.
Considering equation \eqref{3.1} at the point $(x_i,y_j,t_{n-1+\sigma})$, we have
\begin{align}
\nonumber
&{}D_{t}^{\omega(\alpha)}u(x_i,y_j,t_{n-1+\sigma})=K_1\frac{\partial^{\beta}u(x_i,y_j,t_{n-1+\sigma})}
{\partial| x|^{\beta}}+K_2\frac{\partial^{\gamma}u(x_i,y_j,t_{n-1+\sigma})}{\partial|
y|^{\gamma}}+f_{ij}^{n-1+\sigma},\\
&\qquad\qquad\qquad\qquad\qquad\qquad (i,j)\in\omega,~1\leq n\leq N.\label{3.4}
\end{align}
Using Lemma \ref{lemma2.1} and Lemma \ref{lemma2.3}, we get
\begin{align}
{}D_{t}^{\omega(\alpha)}u(x_i,y_j,t_{n-1+\sigma})=
\sum\limits_{k=0}^{n-1}\hat{c}_k^{(n)}\left(U_{ij}^{n-k}-U_{ij}^{n-k-1}\right)+\mathcal{O}\left(\tau^{3-\alpha_m}+\Delta\alpha^2\right).\label{3.5}
\end{align}

Moreover, by applying the second-order linear interpolation formula to the Riesz derivative on the right side of \eqref{3.4} and using Lemma \ref{lemma2.4}, we obtain
\begin{align}\label{3.6}
\nonumber
&\frac{\partial^{\beta}u(x_i,y_j,t_{n-1+\sigma})}{\partial| x|^{\beta}}\\
\nonumber
=&\sigma\frac{\partial^{\beta}u(x_i,y_j,t_n)}{\partial|
x|^{\beta}}+\left(1-\sigma\right)\frac{\partial^{\beta}u(x_i,y_j,t_{n-1})}{\partial|
x|^{\beta}}+\mathcal{O}(\tau^{2})\\
\nonumber
=&\sigma\left[-h_1^{-\beta}\sum\limits_{k=i-M_1}^ig_k^{(\beta)}U_{i-k,j}^n\right]+(1-\sigma)
\left[-h_1^{-\beta}\sum\limits_{k=i-M_1}^ig_k^{(\beta)}U_{i-k,j}^{n-1}\right]+\mathcal{O}(h_1^2+\tau^{2})\\
=&-h_1^{-\beta}\sum\limits_{k=i-M_1}^ig_k^{(\beta)}\left[\sigma U_{i-k,j}^n+(1-\sigma)U_{i-k,j}^{n-1}\right]+\mathcal{O}(h_1^2+\tau^{2}).
\end{align}
Similarly, we can get
\begin{align}\label{3.7}
\frac{\partial^{\gamma}u(x_i,y_j,t_{n-1+\sigma})}{\partial| y|^{\gamma}}
=-h_2^{-\gamma}\sum\limits_{k=j-M_2}^jg_k^{(\gamma)}\left[\sigma U_{i,j-k}^n+(1-\sigma)U_{i,j-k}^{n-1}\right]+\mathcal{O}(h_2^2+\tau^{2}).
\end{align}
By substituting \eqref{3.5}-\eqref{3.7} into \eqref{3.4}, we can get
\begin{align}
\nonumber
&\sum\limits_{k=0}^{n-1}\hat{c}_k^{(n)}\left(U_{ij}^{n-k}-U_{ij}^{n-k-1}\right)\\
\nonumber
=&-K_1h_1^{-\beta}\sum\limits_{k=i-M_1}^ig_k^{(\beta)}\left[\sigma U_{i-k,j}^n+(1-\sigma)U_{i-k,j}^{n-1}\right]\\
\nonumber
&-K_2h_2^{-\gamma}\sum\limits_{k=j-M_2}^jg_k^{(\gamma)}\left[\sigma U_{i,j-k}^n+(1-\sigma)U_{i,j-k}^{n-1}\right]\\
&+f_{ij}^{n-1+\sigma}+S_{ij}^n,\quad(i,j)\in\omega,~1\leq n\leq N,\label{3.8}
\end{align}
where there exists a positive constant $c_2$ such that
\begin{align}
\mid S_{ij}^n\mid\leq c_2\left(h_1^2+h_2^2+\tau^2+\Delta\alpha^2\right),\quad (i,j)\in\omega,~1\leq n\leq N.\label{3.9}
\end{align}

Notice the initial and boundary conditions \eqref{3.2}-\eqref{3.3}, and we have
\begin{align}
&U_{ij}^n=0,\quad (i,j)\in\partial\omega,~0\leq n\leq N,\label{3.10}\\
&U_{ij}^0=\phi(x_i,y_j),\quad (i,j)\in\omega.\label{3.11}
\end{align}
Thus, by neglecting the small term $S_{ij}^n$ in \eqref{3.8} and replacing the exact solution $U_{ij}^n$ with the numerical ones $u_{ij}^k$ in \eqref{3.8} and \eqref{3.10}-\eqref{3.11}, we can get the difference scheme for solving \eqref{3.1}-\eqref{3.3} as follows:
\begin{align}
\nonumber
&\sum\limits_{k=0}^{n-1}\hat{c}_k^{(n)}\left(u_{ij}^{n-k}-u_{ij}^{n-k-1}\right)\\
\nonumber
=&-K_1h_1^{-\beta}\sum\limits_{k=i-M_1}^ig_k^{(\beta)}\left[\sigma u_{i-k,j}^n+(1-\sigma)u_{i-k,j}^{n-1}\right]\\
\nonumber
&-K_2h_2^{-\gamma}\sum\limits_{k=j-M_2}^jg_k^{(\gamma)}\left[\sigma u_{i,j-k}^n+(1-\sigma)u_{i,j-k}^{n-1}\right]\\
&+f_{ij}^{n-1+\sigma},\quad(i,j)\in\omega,~1\leq n\leq N,\label{3.12}\\
&u_{ij}^n=0,\quad (i,j)\in\partial\omega,~0\leq n\leq N,\label{3.13}\\
&u_{ij}^0=\phi(x_i,y_j),\quad (i,j)\in\omega.\label{3.14}
\end{align}

\subsection{Solvability, stability and convergence analysis}
\quad In this subsection, we show that the difference scheme \eqref{3.12}-\eqref{3.14} obtained in Section \ref{section3.1} is uniquely solvable, unconditionally stable and convergent with the order of $\mathcal{O}(h_1^2+h_2^2+\tau^2+\Delta\alpha^2)$.

Let
\begin{align*}
V_h=\{v~|~v=\{v_{ij}~|~(i,j)\in\bar{\omega}\}\},\quad \hat{V}_h=\{v~|~v\in V_h;~v_{ij}=0~when~(i,j)\in\partial\omega\}.
\end{align*}
For any $v,~w\in \hat{V}_h$, the discrete inner product and the corresponding discrete $L_2$-norms are defined as follows:
\begin{align*}
(v,w)=h_1h_2\sum\limits_{i=1}^{M_1-1}\sum\limits_{j=1}^{M_2-1}v_{ij}w_{ij},\quad and \quad \|v\|=\sqrt{(v,v)}.
\end{align*}

We now work towards showing the unique solvability of difference scheme \eqref{3.12}-\eqref{3.14}. The desired result is reported by the following theorem.
\begin{theorem}
The difference scheme \eqref{3.12}-\eqref{3.14} is uniquely solvable.
\label{th3.1}
\end{theorem}
\begin{proof}
Let
$$u^n=\{u_{ij}^n~|~(i,j)\in\bar{\omega}\}.$$

According to \eqref{3.13}-\eqref{3.14}, the value of $u^0$ is determined. Now suppose that $\{u^k~|~0\leq k\leq n-1\}$ has been determined. According to \eqref{3.12} and \eqref{3.13}, we get a linear system of equations with respect to $u^n$. Then we only need to prove that the corresponding homogeneous linear system
\begin{align}
&\hat{c}_0^{(n)}u_{ij}^n
=-K_1\sigma h_1^{-\beta}\sum\limits_{k=i-M_1}^ig_k^{(\beta)}u_{i-k,j}^n
-K_2\sigma h_2^{-\gamma}\sum\limits_{k=j-M_2}^jg_k^{(\gamma)}u_{i,j-k}^n,\quad (i,j)\in\omega,\label{3.15}\\
&u_{ij}^n=0,\quad (i,j)\in\partial\omega\label{3.16}
\end{align}
only has solution of 0.

We first rewrite the equation \eqref{3.15} as follows:
\begin{align}
\nonumber
&\left[\hat{c}_0^{(n)}+K_1\sigma h_1^{-\beta}g_0^{(\beta)}+K_2\sigma h_2^{-\gamma}g_0^{(\gamma)}\right]u_{ij}^n\\
=&K_1\sigma h_1^{-\beta}\sum\limits_{k=i-M_1\atop k\neq0}^i\left(-g_k^{(\beta)}\right)u_{i-k,j}^n
+K_2\sigma h_2^{-\gamma}\sum\limits_{k=j-M_2\atop k\neq0}^j\left(-g_k^{(\gamma)}\right)u_{i,j-k}^n, \quad (i,j)\in\omega.\label{3.17}
\end{align}
Let $\|u^n\|_\infty=\mid u_{i_n,j_n}^n\mid$, where $(i_n,j_n)\in\omega$. We consider the equation \eqref{3.17} with $(i,j)=(i_n,j_n)$ and take absolute values on both sides of the equation. Noticing that the coefficients $K_1>0,~K_2>0$, based on Lemma \ref{lemma2.5} and using triangle inequality, we have
\begin{align*}
\nonumber
&\left[\hat{c}_0^{(n)}+K_1\sigma h_1^{-\beta}g_0^{(\beta)}+K_2\sigma h_2^{-\gamma}g_0^{(\gamma)}\right]\parallel u^n\parallel_\infty\\
\nonumber
=&K_1\sigma h_1^{-\beta}\sum\limits_{k=i_n-M_1\atop k\neq0}^{i_n}\left(-g_k^{(\beta)}\right)|u_{i_n-k,j_n}^n|
+K_2\sigma h_2^{-\gamma}\sum\limits_{k=j_n-M_2\atop k\neq0}^{j_n}\left(-g_k^{(\gamma)}\right)|u_{i_n,j_n-k}^n|\\
\nonumber
\leq&K_1\sigma h_1^{-\beta}\sum\limits_{k=i_n-M_1\atop k\neq0}^{i_n}\left(-g_k^{(\beta)}\right)\parallel u^n\parallel_\infty+K_2\sigma h_2^{-\gamma}\sum\limits_{k=j_n-M_2\atop k\neq0}^{j_n}\left(-g_k^{(\gamma)}\right)\parallel u^n\parallel_\infty\\
\leq&\left[K_1\sigma h_1^{-\beta}g_0^{(\beta)}+K_2\sigma h_2^{-\gamma}g_0^{(\gamma)}\right]\parallel u^n\parallel_\infty.
\end{align*}

Therefore, we get $\parallel u^n\parallel_\infty=0$, which indicates that the homogeneous linear equations \eqref{3.15}-\eqref{3.16} only have solution 0. According to the mathematical induction, the difference scheme \eqref{3.12}-\eqref{3.14} is uniquely solvable.
\end{proof}

We will now discuss the unconditional stability of the difference scheme \eqref{3.12}-\eqref{3.14} with respect to the initial value and the inhomogeneous term $f(x,y,t)$. 
\begin{theorem}
 Let $\{u_{ij}^n~|~(i,j)\in\bar{\omega},~0\leq n\leq N\}$ be the solution of the difference scheme \eqref{3.12}-\eqref{3.14}. We have
\begin{align*}
\parallel u^n\parallel^2\leq&\parallel u^0\parallel^2\\
&+\frac{1}{4}\left[\frac{(2L_1)^{\beta}}{K_1c_*^{(\beta)}}+\frac{(2L_2)^{\gamma}}{K_2c_*^{(\gamma)}}
\right]\frac{1}{\sum\limits_{r=0}^{m}\frac{\lambda_r}{T^{\alpha_r}{\Gamma(1-\alpha_r)}}}\max\limits_{1\leq l\leq n}\parallel f^{l-1+\sigma}\parallel^2,\quad 1\leq n\leq N,
\end{align*}
where
$$\parallel f^{l-1+\sigma}\parallel^2=h_1h_2\sum\limits_{i=1}^{M_1-1}\sum\limits_{j=1}^{M_2-1}
\left(f_{ij}^{l-1+\sigma}\right)^2.$$
\label{th3.2}
\end{theorem}

\begin{proof}
By multiplying \eqref{3.12} by $h_1h_2[\sigma u_{ij}^n+(1-\sigma)u_{ij}^{n-1}]$ and summing up $(i,j)$ with respect to $\omega$, we get
\begin{equation}\label{3.18}
 \begin{split}
\nonumber
&\sum\limits_{k=0}^{n-1}\hat{c}_k^{(n)}h_1h_2\sum\limits_{i=1}^{M_1-1}\sum\limits_{j=1}^{M_2-1}
\left(u_{ij}^{n-k}-u_{ij}^{n-k-1}\right)\left[\sigma u_{ij}^n+(1-\sigma)u_{ij}^{n-1}\right]\\
\nonumber
=&K_1h_2\sum\limits_{j=1}^{M_2-1}\big\{-h_1^{-\beta}h_1\sum\limits_{i=1}^{M_1-1}\sum\limits_{k=i-M_1}^ig_k^{(\beta)}\big[\sigma u_{i-k,j}^n\\
\nonumber
&\qquad\qquad\quad +(1-\sigma)u_{i-k,j}^{n-1}\big]\left[\sigma u_{ij}^n+(1-\sigma)u_{ij}^{n-1}\right]\big\}\\
\nonumber
&+K_2h_1\sum\limits_{i=1}^{M_1-1}\big\{-h_2^{-\gamma}h_2\sum\limits_{j=1}^{M_2-1} \sum\limits_{k=j-M_2}^jg_k^{(\gamma)}\big[\sigma u_{i,j-k}^n\\
\nonumber
&\qquad\qquad\quad +(1-\sigma)u_{i,j-k}^{n-1}\big]\left[\sigma u_{ij}^n+(1-\sigma)u_{ij}^{n-1}\right]\big\}\\
&+h_1h_2\sum\limits_{i=1}^{M_1-1}\sum\limits_{j=1}^{M_2-1}f_{ij}^{n-1+\sigma}\left[\sigma u_{ij}^n+(1-\sigma)u_{ij}^{n-1}\right].
\end{split}
\end{equation}
According to Lemma \ref{lemma2.7}, it follows that
\begin{align}\label{3.19}
\nonumber
&\sum\limits_{k=0}^{n-1}\hat{c}_k^{(n)}h_1h_2\sum\limits_{i=1}^{M_1-1}\sum\limits_{j=1}^{M_2-1}
\left(u_{ij}^{n-k}-u_{ij}^{n-k-1}\right)\left[\sigma u_{ij}^n+(1-\sigma)u_{ij}^{n-1}\right]\\
\geq&\frac{1}{2}\sum\limits_{k=0}^{n-1}\hat{c}_k^{(n)}\left(\parallel u^{n-k}\parallel^2-\parallel u^{n-k-1}\parallel^2\right).
\end{align}
Using Lemma \ref{lemma2.8}, we obtain
\begin{align}\label{3.20}
\nonumber
&-h_1^{-\beta}h_1\sum\limits_{i=1}^{M_1-1} \sum\limits_{k=i-M_1}^ig_k^{(\beta)}\left[\sigma u_{i-k,j}^n+(1-\sigma)u_{i-k,j}^{n-1}\right]\left[\sigma u_{ij}^n+(1-\sigma)u_{ij}^{n-1}\right]\\
\leq&-c_*^{(\beta)}(2L_1)^{-\beta}h_1\sum\limits_{i=1}^{M_1-1}\left[\sigma u_{ij}^n+(1-\sigma)u_{ij}^{n-1}\right]^2
\end{align}
and
\begin{align}\label{3.21}
\nonumber
&-h_2^{-\gamma}h_2\sum\limits_{j=1}^{M_2-1}\sum\limits_{k=j-M_2}^jg_k^{(\gamma)}\left[\sigma u_{i,j-k}^n+(1-\sigma)u_{i,j-k}^{n-1}\right]\left[\sigma u_{ij}^n+(1-\sigma)u_{ij}^{n-1}\right]\\
\leq&-c_*^{(\gamma)}(2L_2)^{-\gamma}h_2\sum\limits_{j=1}^{M_2-1}\left[\sigma u_{ij}^n+(1-\sigma)u_{ij}^{n-1}\right]^2.
\end{align}
By substituting \eqref{3.19}-\eqref{3.21} into \eqref{3.18}, we get
\begin{align}
\nonumber
&\frac{1}{2}\sum\limits_{k=0}^{n-1}\hat{c}_k^{(n)}\left(\parallel u^{n-k}\parallel^2-\parallel u^{n-k-1}\parallel^2\right)\\
\nonumber
\leq&-K_1c_*^{(\beta)}(2L_1)^{-\beta}h_1h_2\sum\limits_{i=1}^{M_1-1}\sum\limits_{j=1}^{M_2-1}
\left[\sigma u_{ij}^n+(1-\sigma)u_{ij}^{n-1}\right]^2\\
\nonumber
&-K_2c_*^{(\gamma)}(2L_2)^{-\gamma}h_1h_2\sum\limits_{i=1}^{M_1-1}\sum\limits_{j=1}^{M_2-1}
\left[\sigma u_{ij}^n+(1-\sigma)u_{ij}^{n-1}\right]^2\\
\nonumber
&+h_1h_2\sum\limits_{i=1}^{M_1-1}\sum\limits_{j=1}^{M_2-1}f_{ij}^{n-1+\sigma}\left[\sigma u_{ij}^n+(1-\sigma)u_{ij}^{n-1}\right]\\
\nonumber
\leq&-K_1c_*^{(\beta)}(2L_1)^{-\beta}\parallel \sigma u^n+(1-\sigma)u^{n-1}\parallel^2\\
\nonumber
&-K_2c_*^{(\gamma)}(2L_2)^{-\gamma}\parallel \sigma u^n+(1-\sigma)u^{n-1}\parallel^2\\
\nonumber
&+\parallel f^{n-1+\sigma}\parallel\cdot\parallel \sigma u^n+(1-\sigma)u^{n-1}\parallel\\
\leq&\frac{1}{16}\left[\frac{(2L_1)^{\beta}}{K_1c_*^{(\beta)}}+\frac{(2L_2)^{\gamma}}{K_2c_*^{(\gamma)}}
\right]\parallel f^{n-1+\sigma}\parallel^2,\quad 1\leq n\leq N.\label{3.22}
\end{align}

With the use of Lemma \ref{lemma2.6}, we have
\begin{align}
\hat{c}_{n-1}^{(n)}\geq
\sum\limits_{r=0}^{m}\lambda_r\frac{\tau^{-\alpha_r}}{\Gamma(2-\alpha_r)}\cdot\frac{1-\alpha_r}{2}
(n-1+\sigma)^{-\alpha_r}\geq \frac{1}{2}\sum\limits_{r=0}^{m}\frac{\lambda_r}{T^{\alpha_r}{\Gamma(1-\alpha_r)}}.\label{3.23}
\end{align}
By combining \eqref{3.22} and \eqref{3.23}, we arrive at the following inequality:
\begin{align*}
\nonumber
&\hat{c}_0^{(n)}\parallel u^{n}\parallel^2\\
\leq&\sum\limits_{k=1}^{n-1}\left(\hat{c}_{k-1}^{(n)}-\hat{c}_k^{(n)}\right)\parallel u^{n-k}\parallel^2+\hat{c}_{n-1}^{(n)}\parallel u^{0}\parallel^2
+\frac{1}{8}\left[\frac{(2L_1)^{\beta}}{K_1c_*^{(\beta)}}+\frac{(2L_2)^{\gamma}}{K_2c_*^{(\gamma)}}
\right]\parallel f^{n-1+\sigma}\parallel^2\\
\leq&\sum\limits_{k=1}^{n-1}\left(\hat{c}_{k-1}^{(n)}-\hat{c}_k^{(n)}\right)\parallel u^{n-k}\parallel^2\\
&+\hat{c}_{n-1}^{(n)}\left\{\parallel u^{0}\parallel^2
+\frac{1}{4}\left[\frac{(2L_1)^{\beta}}{K_1c_*^{(\beta)}}+\frac{(2L_2)^{\gamma}}{K_2c_*^{(\gamma)}}
\right]\frac{1}{\sum\limits_{r=0}^{m}\frac{\lambda_r}{T^{\alpha_r}{\Gamma(1-\alpha_r)}}}\parallel f^{n-1+\sigma}\parallel^2\right\},
\end{align*}
where $1\leq n\leq N$.
Applying the mathematical induction method to the above inequality, we can get the conclusion of Theorem \ref{th3.2}. This completes the proof.
\end{proof}
Now we will prove that the proposed difference scheme \eqref{3.12}-\eqref{3.14} is unconditionally convergent in $L_2$-norm with the quadratic-order accuracy in time, space and distributed-order integral variables.

Suppose that $\{U_{ij}^n~|~(i,j)\in\bar{\omega},~0\leq n\leq N\}$ is the exact solution of the system \eqref{3.1}-\eqref{3.3} and $\{u_{ij}^n~|~(i,j)\in\bar{\omega},~0\leq n\leq N\}$ is the numerical solution of the difference scheme \eqref{3.12}-\eqref{3.14}.
 Let $e_{ij}^n=U_{ij}^n-u_{ij}^n~((i,j)\in\bar{\omega},~0\leq n\leq N)$.

By subtracting \eqref{3.12}-\eqref{3.14} from \eqref{3.8}, \eqref{3.10}-\eqref{3.11}, respectively, we can get the following error equations:
\begin{align*}
&\sum\limits_{k=0}^{n-1}\hat{c}_k^{(n)}\left(e_{ij}^{n-k}-e_{ij}^{n-k-1}\right)\\
=&-K_1h_1^{-\beta}\sum\limits_{k=i-M_1}^ig_k^{(\beta)}\left[\sigma e_{i-k,j}^n+(1-\sigma)e_{i-k,j}^{n-1}\right]\\
&-K_2h_2^{-\gamma}\sum\limits_{k=j-M_2}^jg_k^{(\gamma)}\left[\sigma e_{i,j-k}^n+(1-\sigma)e_{i,j-k}^{n-1}\right]\\
&+S_{ij}^n,\quad(i,j)\in\omega,~1\leq n\leq N,\\
&e_{ij}^n=0,\quad (i,j)\in\partial\omega,~0\leq n\leq N,\\
&e_{ij}^0=0,\quad (i,j)\in\omega.
\end{align*}
Applying the conclusion of Theorem \ref{th3.2} and noticing \eqref{3.9}, we have
\begin{align*}
&\parallel e^n\parallel^2\\
\leq&\frac{1}{4}\left[\frac{(2L_1)^{\beta}}{K_1c_*^{(\beta)}}
+\frac{(2L_2)^{\gamma}}{K_2c_*^{(\gamma)}}\right]\frac{1}{\sum\limits_{r=0}^{m}
\frac{\lambda_r}{T^{\alpha_r}{\Gamma(1-\alpha_r)}}}\max\limits_{1\leq l\leq n}\parallel S^l\parallel^2\\
\leq&\frac{1}{4}\left[\frac{(2L_1)^{\beta}}{K_1c_*^{(\beta)}}
+\frac{(2L_2)^{\gamma}}{K_2c_*^{(\gamma)}}\right]\frac{1}{\sum\limits_{r=0}^{m}
\frac{\lambda_r}{T^{\alpha_r}{\Gamma(1-\alpha_r)}}}
\left[c_2\left(h_1^2+h_2^2+\tau^2+\Delta\alpha^2\right)\right]^2L_1L_2.
\end{align*}
By extracting the square root on both sides of the above equation, we acquire
\begin{align*}
\parallel e^n\parallel
\leq\frac{c_2}{2}\sqrt{\left[\frac{(2L_1)^{\beta}}{K_1c_*^{(\beta)}}
+\frac{(2L_2)^{\gamma}}{K_2c_*^{(\gamma)}}\right]\frac{L_1L_2}{\sum\limits_{r=0}^{m}
\frac{\lambda_r}{T^{\alpha_r}{\Gamma(1-\alpha_r)}}}}
\left(h_1^2+h_2^2+\tau^2+\Delta\alpha^2\right),
\end{align*}
where $1\leq n\leq N$.

Now, we arrive at the following result.
\begin{theorem}
Suppose that the continuous problem \eqref{3.1}-\eqref{3.3} has a smooth solution $u(x,y,t)\in C^{(5,5,3)}(\Omega\times[0,T])$, and let $u_{ij}^n$ be the solution of the difference scheme \eqref{3.12}-\eqref{3.14}. it holds that
\begin{align*}
\parallel e^n\parallel
\leq\frac{c_2}{2}\sqrt{\left[\frac{(2L_1)^{\beta}}{K_1c_*^{(\beta)}}
+\frac{(2L_2)^{\gamma}}{K_2c_*^{(\gamma)}}\right]\frac{L_1L_2}{\sum\limits_{r=0}^{m}
\frac{\lambda_r}{T^{\alpha_r}{\Gamma(1-\alpha_r)}}}}
\left(h_1^2+h_2^2+\tau^2+\Delta\alpha^2\right),
\end{align*}
\label{th3.3}
\end{theorem}
where $1\leq n\leq N$.

\subsection{Fast solution techniques with circulant preconditioner}\label{section3.3}
\quad Let
$$\mathbf{u}^n=(u_{1,1}^n,\cdots,u_{M_1-1,1}^n,u_{1,2}^n,\cdots,u_{M_1-1,2}^n,u_{1,M_2-1}^n,\cdots,u_{M_1-1,M_2-1}^n)^T,$$
$$\mathbf{f}^n=(f_{1,1}^n,\cdots,f_{M_1-1,1}^n,f_{1,2}^n,\cdots,f_{M_1-1,2}^n,f_{1,M_2-1}^n,\cdots,f_{M_1-1,M_2-1}^n)^T.$$
Then the implicit difference scheme \eqref{3.12} can be rewritten in the matrix form
\begin{equation}\label{3.24}
 M^{n}\mathbf{u}^{n}=p^{n-1},\quad n=1,2,\ldots,N,
\end{equation}
in which
 \begin{equation}\label{3.25}
 M^{n}=\hat{c}_0^{(n)}I_{3}+\sigma K_1h_1^{-\beta}I_2\otimes G_{\beta}+\sigma K_2h_2^{-\gamma}G_{\gamma}\otimes I_1,
 \end{equation}
 and
\begin{align*}
 p^{n-1}=&-(1-\sigma)[K_1h_1^{-\beta}I_2\otimes G_{\beta}
 +K_2h_2^{-\gamma}G_{\gamma}\otimes I_1]\mathbf{u}^{n-1}\\
 &+\sum\limits_{k=1}^{n-1}(\hat{c}_{k-1}^{(n)}-\hat{c}_k^{(n)})\mathbf{u}^{n-k}+\hat{c}_{n-1}^{(n)}u^0+\mathbf{f}^{n-1+\sigma},
 \end{align*}
where $\otimes$ denotes the Kronecker product, $I_1$, $I_2$ and $I_3$ are identity matrices with orders of $M_1-1$, $M_2-1$ and $(M_1-1)(M_2-1)$, respectively. $G_\beta\in \mathbb{R}^{(M_1-1)\times (M_1-1)}$ and $G_\gamma\in \mathbb{R}^{(M_2-1)\times (M_2-1)}$ are Toeplitz matrices and have forms as \eqref{2.31}.

The following lemma guarantees the invertibility of the coefficient matrix $M^{n}$ in \eqref{3.25}.

\begin{lemma}
 The coefficient matrix
\begin{equation*}
 M^{n}=\hat{c}_0^{(n)}I_{3}+\sigma K_1h_1^{-\beta}I_2\otimes G_{\beta}+\sigma K_2h_2^{-\gamma}G_{\gamma}\otimes I_1,
 \end{equation*}
of the linear system \eqref{3.24} is a symmetric positive definite matrix.
\label{lemma3.1}
\end{lemma}

\begin{proof}
According to Lemma \ref{lemma2.5} and the definitions of the matrices $G_\beta$ and $G_{\gamma}$, one can prove that $G_{\beta}$ and $G_{\gamma}$ are symmetric positive definite matrices. Therefore, the matrices $I_2\otimes G_{\beta}$ and $G_{\gamma}\otimes I_1$ are symmetric positive definite matrices. Given that $\hat{c}_0^{(n)}>0$ and $K>0$, it is easy to show that the matrix $M^{n}$, which is defined by \eqref{3.25}, is also a symmetric positive definite matrix.
\end{proof}
We also use the CG method for solving the linear system \eqref{3.24}. In order to improve the performance
and reliability of the CG method, the preconditioning techniques are exploited. We refer to the
coefficient matrix $M^{n}$ as a block Toeplitz matrix with Toeplitz blocks (BTTB) \cite{chan2007introduction},
Therefore the following level-2 circulant preconditioner which is a block circulant matrix with circulant blocks (BCCB) is considered:
\begin{equation*}
 C_2^{n}=\hat{c}_0^{(n)}I_{3}+\sigma K_1h_1^{-\beta}I_2\otimes c(G_{\beta})+\sigma K_2h_2^{-\gamma}c(G_{\gamma})\otimes I_1.
 \end{equation*}

Similarly, we discuss the properties of the circulant preconditioner $C_2^{n}$ as follows.

\begin{lemma}
 The level-2 circulant preconditioner
\begin{equation*}
 C_2^{n}=\hat{c}_0^{(n)}I_{3}+\sigma K_1h_1^{-\beta}I_2\otimes c(G_{\beta})+\sigma K_2h_2^{-\gamma}c(G_{\gamma})\otimes I_1.
\end{equation*}
is a symmetric positive definite matrix.
\label{lemma3.2}
\end{lemma}
\begin{proof}
According to the proof of Lemma \ref{lemma2.10},  it is easy to see that $c(G_{\beta})$ and $c(G_{\gamma})$ are symmetric positive definite matrices. Then, as similar to the proof of Lemma \ref{lemma3.1}, we can prove that the level-2 circulant preconditioner $C_2^{n}$ is a symmetric positive definite matrix.
\end{proof}
According to Lemma \ref{lemma3.2}, we can know that the preconditioner $C_2^n$ is nonsingular. Theoretically, for the BCCB matrix $C_2^{n}$, the spectrum of $(C_2^{n})^{-1}M^{n}$ is clustered around 1 except for at most $\mathcal{O}(M_1-1)+\mathcal{O}(M_2-1)$ outlying eigenvalues \cite{chan2007introduction}. When the PCG method is used to solve \eqref{3.24}, the convergence rate will be fast. In Section \ref{section4}, we will also present numerical examples to demonstrate the usefulness of the proposed circulant preconditioning $C_2^n$. Thus, the total complexity of the PCG method with preconditioner $C_2^{n}$ for solving the \eqref{3.24} remains $\mathcal{O}((M_1-1)(M_2-1)\log(M_1-1)(M_2-1))$.

\begin{figure}[!hbt]
  \centering
  \subfigure{
    \includegraphics[width = 6 cm]{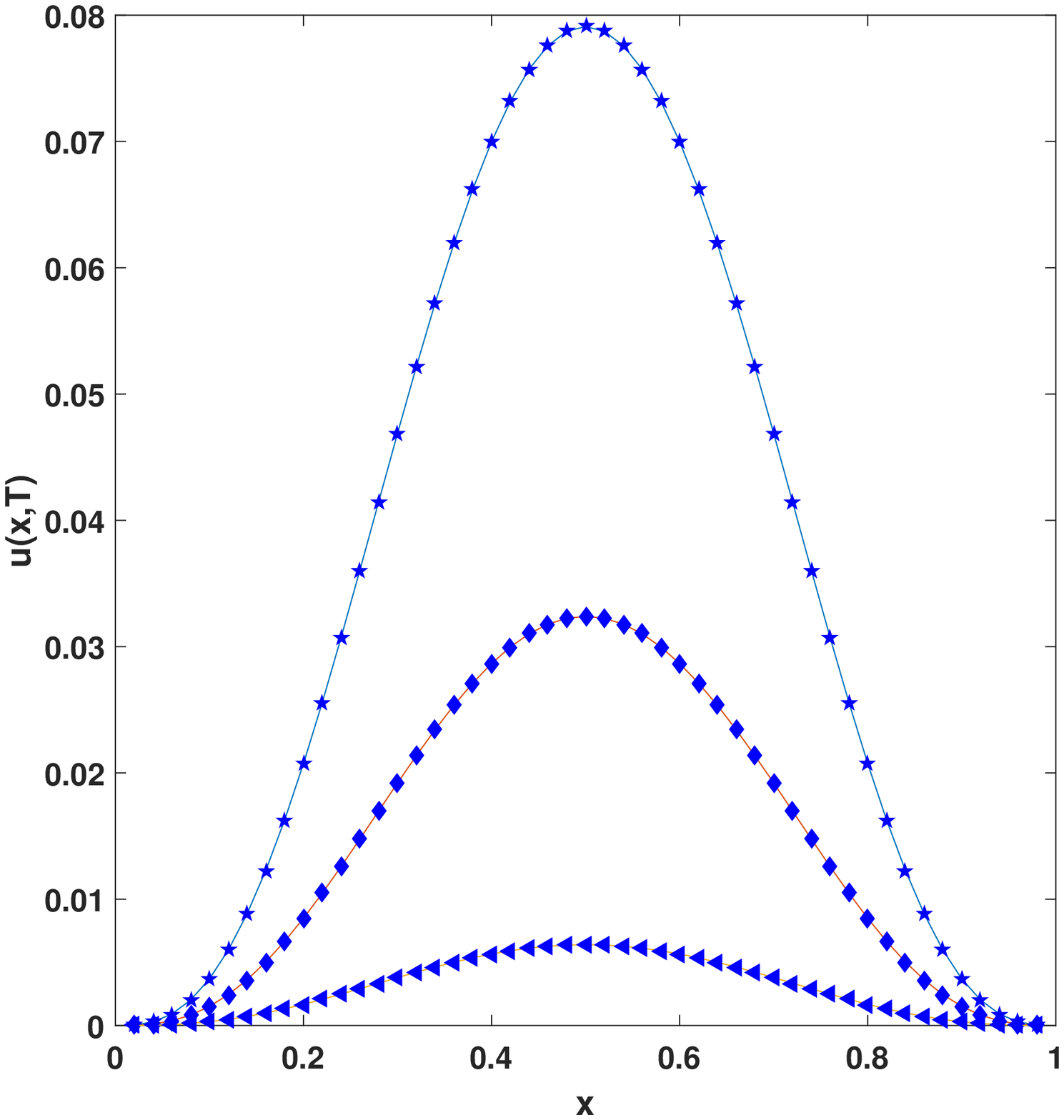}}
     \hspace{1 cm}
  \subfigure{
    \includegraphics[width = 6 cm]{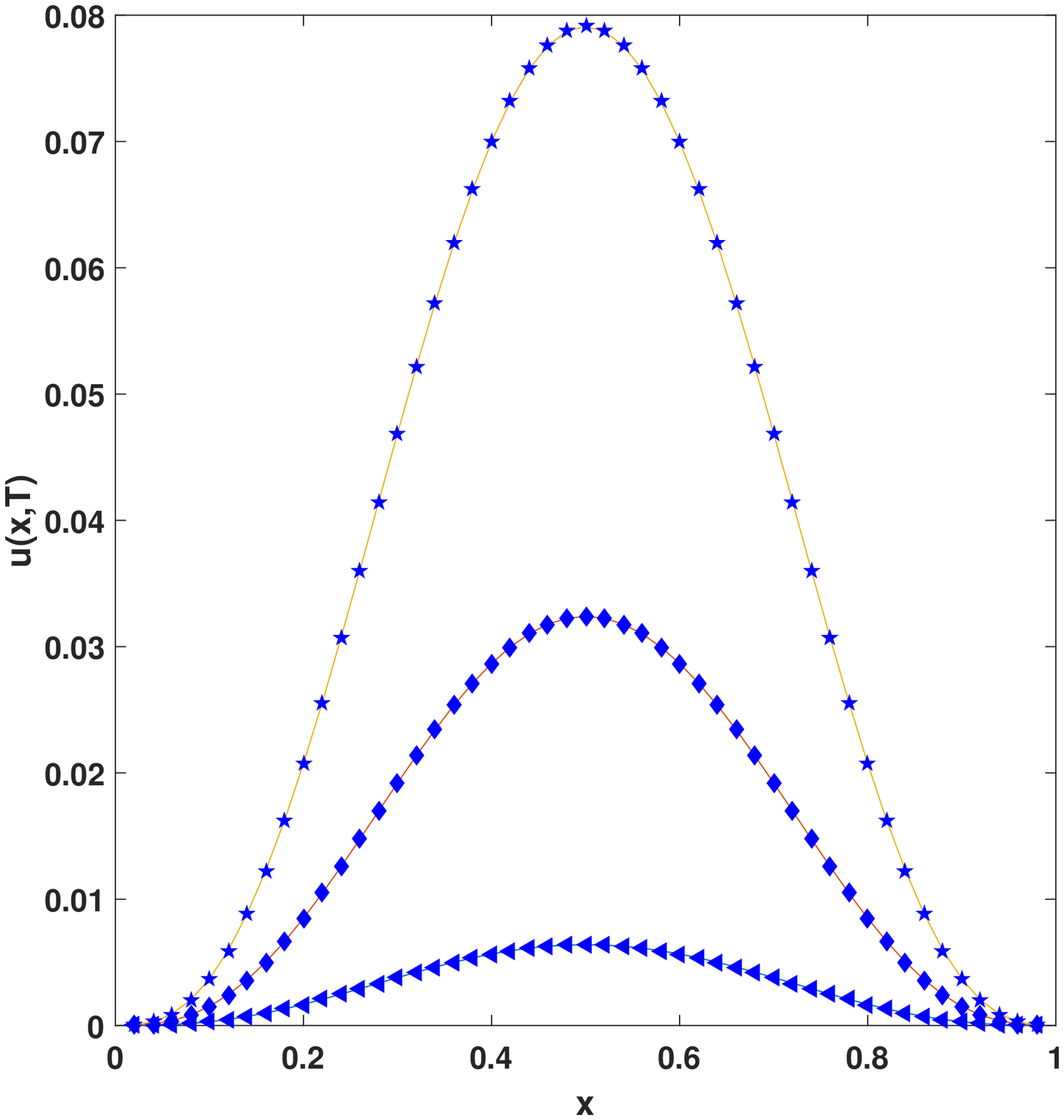}}\\
    (\textbf{a}) $\beta$ = 1.3 \hspace{5cm} (\textbf{b}) $\beta$ = 1.8 \\
  \caption{Exact solutions (lines) and numerical solutions (symbols) of Example \ref{example1}: (\textbf{a}) $\beta$ = 1.3 at $T$ = 1.5 (stars), 1.2 (rhombus), 0.8 (triangles); (\textbf{b}) $\beta$ = 1.8 at $T$ = 1.5 (stars), 1.2 (rhombus), 0.8 (triangles).}
  \label{fig1}
\end{figure}
\begin{table}[t]\small\tabcolsep=5.3pt
\begin{center}
\caption{{\small {Maximum errors and spatial convergence orders of difference scheme \eqref{2.17}-\eqref{2.19} for Example \ref{example1} with $T$ = 1.5; $J$ = 50; $N$ = 1000.}}}
\begin{tabular}{ccccccc}
\hline  & \multicolumn{2}{c}{$\beta=1.2$} & \multicolumn{2}
{c}{$\beta=1.5$} & \multicolumn{2}{c}{$\beta=1.8$} \\
[-2pt] \cmidrule(lr){2-3} \cmidrule(lr){4-5} \cmidrule(lr){6-7} \\ [-11pt]
 $M$ & $e(h,\tau,\Delta\alpha)$ & $rate_h$ & $e(h,\tau,\Delta\alpha)$ &$rate_h$ & $e(h,\tau,\Delta\alpha)$ & $rate_h$\\
\hline
32 & 3.423357e-05 &  -    & 6.057253e-05 &   -    & 9.145302e-05 &   -   \\
64 & 8.665990e-06 &1.9820 & 1.538522e-05 & 1.9771 & 2.313195e-05 & 1.9831 \\
128& 2.165532e-06 &2.0006 & 3.847714e-06 & 1.9995 & 5.789457e-06 & 1.9984 \\
256& 5.410059e-07 &2.0010 & 9.617429e-07 & 2.0003 & 1.449718e-06 & 1.9977 \\
512& 1.353449e-07 &1.9990 & 2.404824e-07 & 1.9997 & 3.637954e-07 & 1.9946 \\
\hline
\end{tabular}
\label{tab1}
\end{center}
\end{table}

\section{Numerical example}\label{section4}
\quad In this section, we carry out numerical examples to demonstrate the second-order accuracy of the proposed difference schemes and the computational efficiency of the preconditioned Krylov subspace methods. At each time level, we employ the Cholesky method, the CG method and the PCG method for solving the resultant linear systems, respectively. The initial guess for all method is chosen as the zero vector and the stopping criterion is ${\|r^{(k)}\|_2}/{\|r^{(0)}\|_2} < 10^{-12}$, where $r^{(k)}$ is the residual vector after $k$ iterations. Number of iterations required for convergence and CPU time of each method are reported. All numerical experiments are performed in MATLAB (R2016a) on a desktop with 16GB RAM, Inter (R) Core (TM) i5-4590 CPU, @3.30GHz.

In Tables 4 and 8, ``CPU(s)" denotes the total CPU time in seconds to solve the linear systems, and ``Iter" denotes the average number of iterations over 10 runs. For the PCG method, we also report the Strang-based circulant preconditioner \cite{lei2013circulant} $S^n$ and the T. Chan-based circulant preconditioner \cite{chan2007introduction} $T^n$. Among them, the circulant preconditioner $S^n$ is shown below, and the circulant preconditioner $T^n$ takes the same form except that we replace the $S$ with $T$.
\begin{equation*}
 S^{n}=\hat{c}_0^{(n)}I+\sigma Kh^{-\beta}s(G)
 \end{equation*}
 and
 \begin{equation*}
 S_2^{n}=\hat{c}_0^{(n)}I_{3}+\sigma K_1h_1^{-\beta}I_2\otimes s(G_{\beta})+\sigma K_2h_2^{-\gamma}s(G_{\gamma})\otimes I_1,
 \end{equation*}
 where $s(\cdot)$ denotes the Strang circulant preconditioner for the Toeplitz matrix.
 More precisely, the first column of the criculant matrix $s(G_\beta)$ is given by
 \begin{center}
$\left(
  \begin{array}{c}
    g_0^{(\beta)} \\
   g_1^{(\beta)} \\
   g_2^{(\beta)} \\
    \vdots \\
    g_{\lfloor\frac{M}{2}\rfloor-1}^{(\beta)}\\
    g_{\lfloor\frac{M}{2}\rfloor+1-M}^{(\beta)}\\
    \vdots \\
    g_{-2}^{(\beta)} \\ \\
    g_{-1}^{(\beta)} \\ \\
  \end{array}
\right)$.
 \end{center}

\begin{table}[t]\small\tabcolsep=5.3pt
\begin{center}
\caption{{\small {Maximum errors and temporal convergence orders of difference scheme \eqref{2.17}-\eqref{2.19} for Example \ref{example1} with $T$ = 1.5; $J$ = 50; $M$ = 1000.}}}
\begin{tabular}{ccccccc}
\hline  & \multicolumn{2}{c}{$\beta=1.2$} & \multicolumn{2}
{c}{$\beta=1.5$} & \multicolumn{2}{c}{$\beta=1.8$} \\
[-2pt] \cmidrule(lr){2-3} \cmidrule(lr){4-5} \cmidrule(lr){6-7} \\ [-11pt]
 $N$  & $e(h,\tau,\Delta\alpha)$ & $rate_\tau$ & $e(h,\tau,\Delta\alpha)$ &$rate_\tau$ & $e(h,\tau,\Delta\alpha)$ & $rate_\tau$\\
\hline
8 & 5.113611e-04 &   -  & 6.193316e-04&   -  & 7.401543e-04&   -   \\
16& 1.294286e-04 &1.9822& 1.590357e-04&1.9614& 1.921523e-04&1.9456 \\
32& 3.228485e-05 &2.0032& 4.014123e-05&1.9862& 4.892436e-05&1.9736 \\
64& 8.019176e-06 &2.0093& 1.004941e-05&1.9980& 1.231641e-05&1.9900 \\
128&2.005007e-06 &1.9998& 2.508026e-06&2.0025& 3.067275e-06&2.0056 \\
\hline
\end{tabular}
\label{tab2}
\end{center}
\end{table}
\begin{table}[t]\small\tabcolsep=5.3pt
\begin{center}
\caption{{\small {Maximum errors and distributed-order integral convergence orders of difference scheme \eqref{2.17}-\eqref{2.19} for Example \ref{example1} with $T$ = 1.5; $M$ = 2000; $N$ = 2000.}}}
\begin{tabular}{ccccccc}
\hline  & \multicolumn{2}{c}{$\beta=1.2$} & \multicolumn{2}
{c}{$\beta=1.5$} & \multicolumn{2}{c}{$\beta=1.8$} \\
[-2pt] \cmidrule(lr){2-3} \cmidrule(lr){4-5} \cmidrule(lr){6-7} \\ [-11pt]
 $J$  & $e(h,\tau,\Delta\alpha)$ & $rate_{\Delta\alpha}$ & $e(h,\tau,\Delta\alpha)$ &$rate_{\Delta\alpha}$ & $e(h,\tau,\Delta\alpha)$ & $rate_{\Delta\alpha}$\\
\hline
2 & 3.774623e-05 &   -  & 3.492953e-05&   -  & 3.152437e-05&   -   \\
4 & 9.457965e-06 &1.9967& 8.747785e-06&1.9975& 7.889510e-06&1.9985 \\
8 & 2.366072e-06 &1.9990& 2.185762e-06&2.0008& 1.967605e-06&2.0035 \\
16& 5.917841e-07 &1.9994& 5.441518e-07&2.0061& 4.862444e-07&2.0167 \\
32& 1.480919e-07 &1.9986& 1.336461e-07&2.0256& 1.158196e-07&2.0698 \\
\hline
\end{tabular}
\label{tab3}
\end{center}
\end{table}

In the following tables, we use ``Chol" as the Cholesky method, ``PCG(S)" as the PCG with the Strang-based preconditioner, ``PCG(T)" as the PCG with the T. Chan-based circulant preconditioner, and ``PCG(C)" as the PCG with the proposed circulant preconditioner.

\begin{example}\label{example1}
Consider the following one-dimensional time distributed-order and Riesz space fractional diffusion problem:
$$\begin{cases}
\int_0^1\Gamma(5-\alpha){}^C_0D_t^{\alpha}u(x,t)d\alpha=\frac{\partial^{\beta}u(x,t)}{\partial\mid
x\mid^{\beta}}+f(x,t),\quad 0<x<1,~0<t\leq T,\\
 u(0,t)=0,\;u(1,t)=0,\quad 0\leq t\leq T,\\
  u(x,0)=0,\quad 0<x<1,
 \end{cases}$$
with $$f(x,t)=f_0(x,t)-ct^4\left[f_1(x,t)-3f_2(x,t)+3f_3(x,t)-f_4(x,t)\right],$$
where $c=-\frac{1}{2\cos(\beta\pi/2)}$, and
\begin{align*}
&f_0(x,t)=24t^3(t-1)x^3(1-x)^3/\ln t,\\
&f_1(x,t)={\Gamma(4)}/{\Gamma(4-\beta)}[x^{3-\beta}+(1-x)^{3-\beta}],\\
&f_2(x,t)={\Gamma(5)}/{\Gamma(5-\beta)}[x^{4-\beta}+(1-x)^{4-\beta}],\\
&f_3(x,t)={\Gamma(6)}/{\Gamma(6-\beta)}[x^{5-\beta}+(1-x)^{5-\beta}],\\
&f_4(x,t)={\Gamma(7)}/{\Gamma(7-\beta)}[x^{6-\beta}+(1-x)^{6-\beta}].\\
\end{align*}
The exact solution of this example is given by $u(x,t)=t^4x^3(1-x)^3$.
\end{example}

Let $e(h,\tau,\Delta\alpha)=\max\limits_{0\leq i\leq M\atop0\leq n\leq N}|u(x_i,t_n,\Delta\alpha)-u_i^n|$, where $u(x_i,t_n,\Delta\alpha)$ and $u_i^n$ are the exact solution and numerical solution with the step sizes $h$, $\tau$ and $\Delta\alpha$, respectively. We define the convergence orders as
$$rate_h=\log_2\frac{e(h,\tau,\Delta\alpha)}{e(h/2,\tau,\Delta\alpha)},\quad
rate_{\tau}=\log_2\frac{e(h,\tau,\Delta\alpha)}{e(h,\tau/2,\Delta\alpha)},\quad
rate_{\Delta\alpha}=\log_2\frac{e(h,\tau,\Delta\alpha)}{e(h,\tau,{\Delta\alpha}/2)}.$$

We take $J=50$, $M=50$, $N=50$. Fig. \ref{fig1} shows a comparison between the exact solutions and numerical solutions of the difference scheme \eqref{2.17}-\eqref{2.19} when solving Example \ref{example1} with different $\beta$ and $T$. The good agreement between numerical solutions with the exact solutions can be clearly seen.

\begin{table}[t]\small\tabcolsep=6.0pt
\begin{center}
\caption{{\small {Comparisons on Example \ref{example1} between the Cholesky method, the CG method, and the PCG method with different circulant preconditioners, where $\beta$ = 1.2, 1.5 and 1.8, $J$ = 50 and $T$ = 1.5.}}}
\begin{tabular}{cccccccccccc}
\\
\hline & & &\multicolumn{1}{c}{$\rm{Chol}$} & \multicolumn{2}
{c}{$\rm{CG}$}& \multicolumn{2}{c}{$\rm{PCG(S)}$} & \multicolumn{2}
{c}{$\rm{PCG(T)}$} & \multicolumn{2}{c}{$\rm{PCG(C)}$} \\
[-2pt] \cmidrule(lr){4-4} \cmidrule(lr){5-6} \cmidrule(lr){7-8} \cmidrule(lr){9-10} \cmidrule(lr){11-12} \\ [-11pt]
 $\beta$  & $M$ & $N$ & $\rm{CPU(s)}$ & $\rm{CPU(s)}$ & $\rm{Iter}$ & $\rm{CPU(s)}$ & $\rm{Iter}$ & $\rm{CPU(s)}$ & $\rm{Iter}$ & $\rm{CPU(s)}$ & $\rm{Iter}$\\
\hline
    &$2^6$ &$2^4$ &0.01  &0.01 &24.0 &0.01 &7.0 &0.01 &7.9 &\textbf{0.01} &\textbf{6.0} \\
    &$2^7$ &$2^5$ &0.02  &0.03 &32.0 &0.02 &7.0 &0.02 &7.9 &\textbf{0.02} &\textbf{6.0} \\
    &$2^8$ &$2^6$ &0.09  &0.11 &40.0 &0.05 &7.0 &0.05 &7.0 &\textbf{0.05} &\textbf{6.0} \\
1.2 &$2^9$ &$2^7$ &0.52  &0.58 &47.0 &0.20 &6.0 &0.22 &7.0 &\textbf{0.20} &\textbf{6.0} \\
    &$2^{10}$&$2^8$&4.00 &1.58 &53.0 &0.45 &6.0 &0.49 &7.0 &\textbf{0.49} &\textbf{7.0} \\
    &$2^{11}$&$2^9$&35.88&10.52&60.0 &2.79 &7.0 &2.79 &7.0 &\textbf{2.80} &\textbf{7.0} \\
  \\
    &$2^6$ &$2^4$ &0.00  &0.01 &31.6 &0.01 &8.0 &0.01 &9.0 &\textbf{0.01} &\textbf{6.0} \\
    &$2^7$ &$2^5$ &0.02  &0.05 &53.0 &0.02 &8.0 &0.02 &9.9 &\textbf{0.02} &\textbf{6.0} \\
    &$2^8$ &$2^6$ &0.08  &0.19 &83.0 &0.05 &7.0 &0.06 &10.0&\textbf{0.05} &\textbf{6.0} \\
1.5 &$2^9$ &$2^7$ &0.51  &1.32 &122.9&0.22 &7.0 &0.26 &10.0&\textbf{0.22} &\textbf{7.0} \\
    &$2^{10}$&$2^8$&4.00 &4.47 &167.0&0.49 &7.0 &0.58 &9.0 &\textbf{0.49} &\textbf{7.0} \\
    &$2^{11}$&$2^9$&35.48&34.27&214.0&2.80 &7.0 &3.28 &9.0 &\textbf{2.80} &\textbf{7.0} \\
 \\
    &$2^6$ &$2^4$ &0.00  &0.01 &32.0 &0.01 &8.0 &0.01 &11.5&\textbf{0.01} &\textbf{6.0}  \\
    &$2^7$ &$2^5$ &0.02  &0.05 &64.0 &0.02 &8.0 &0.03 &13.0&\textbf{0.02} &\textbf{6.0} \\
    &$2^8$ &$2^6$ &0.08  &0.27 &122.0&0.06 &8.0 &0.07 &14.0&\textbf{0.05} &\textbf{6.0} \\
1.8 &$2^9$ &$2^7$ &0.51  &2.30 &224.0&0.23 &8.0 &0.33 &15.0&\textbf{0.23} &\textbf{7.9} \\
    &$2^{10}$&$2^8$&4.00 &10.43&402.0&0.53 &8.0 &0.85 &16.0&\textbf{0.53} &\textbf{8.0} \\
    &$2^{11}$&$2^9$&35.66&107.08&684.0&3.04&8.0 &5.26 &17.1&\textbf{3.04} &\textbf{8.0} \\
\hline
\end{tabular}
\label{tab4}
\end{center}
\end{table}

\begin{figure}[!hbt]
  \centering
  \subfigure{
    \includegraphics[width = 6.7 cm]{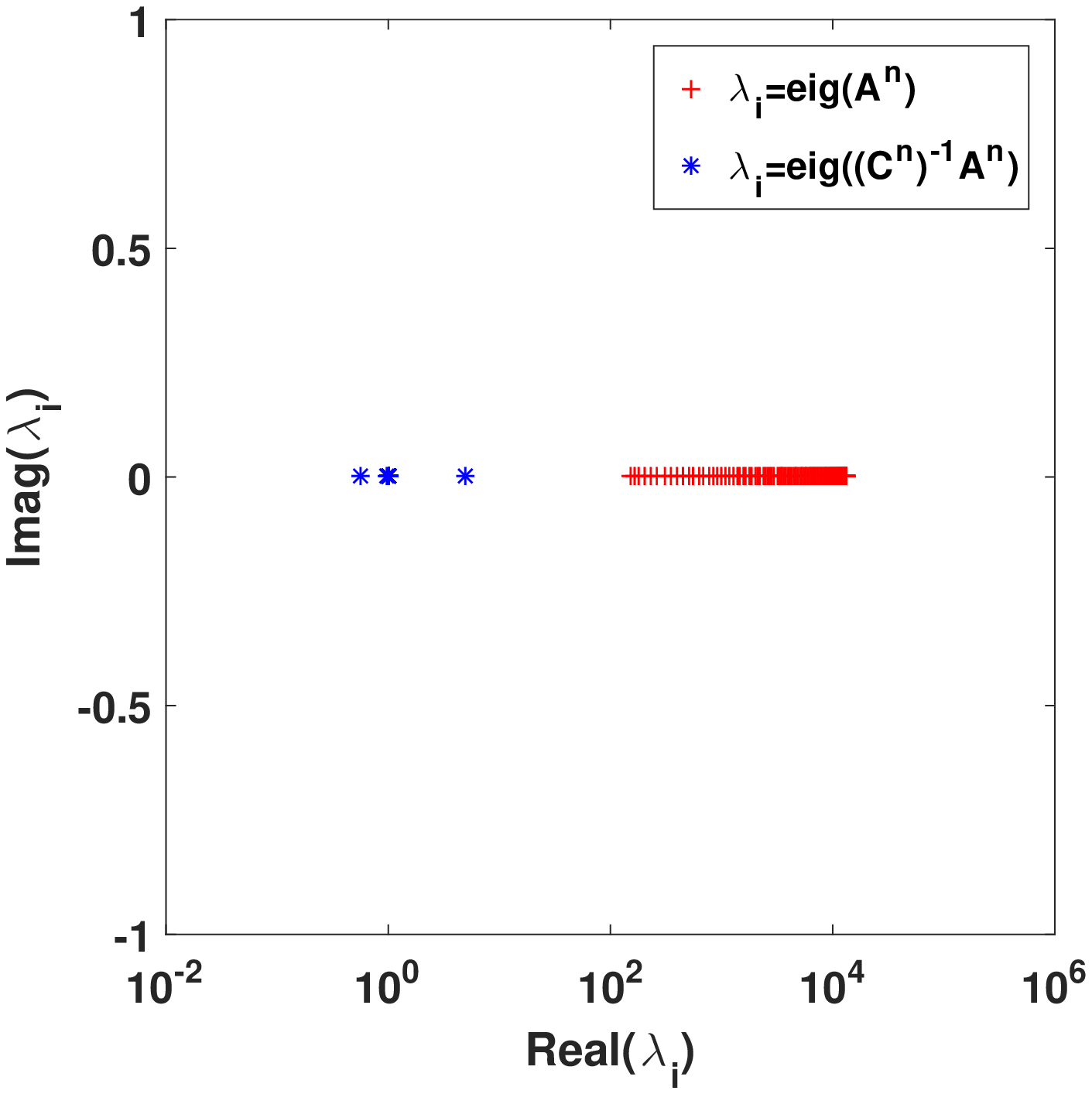}}
     \hspace{1 cm}
  \subfigure{
    \includegraphics[width = 6.7 cm]{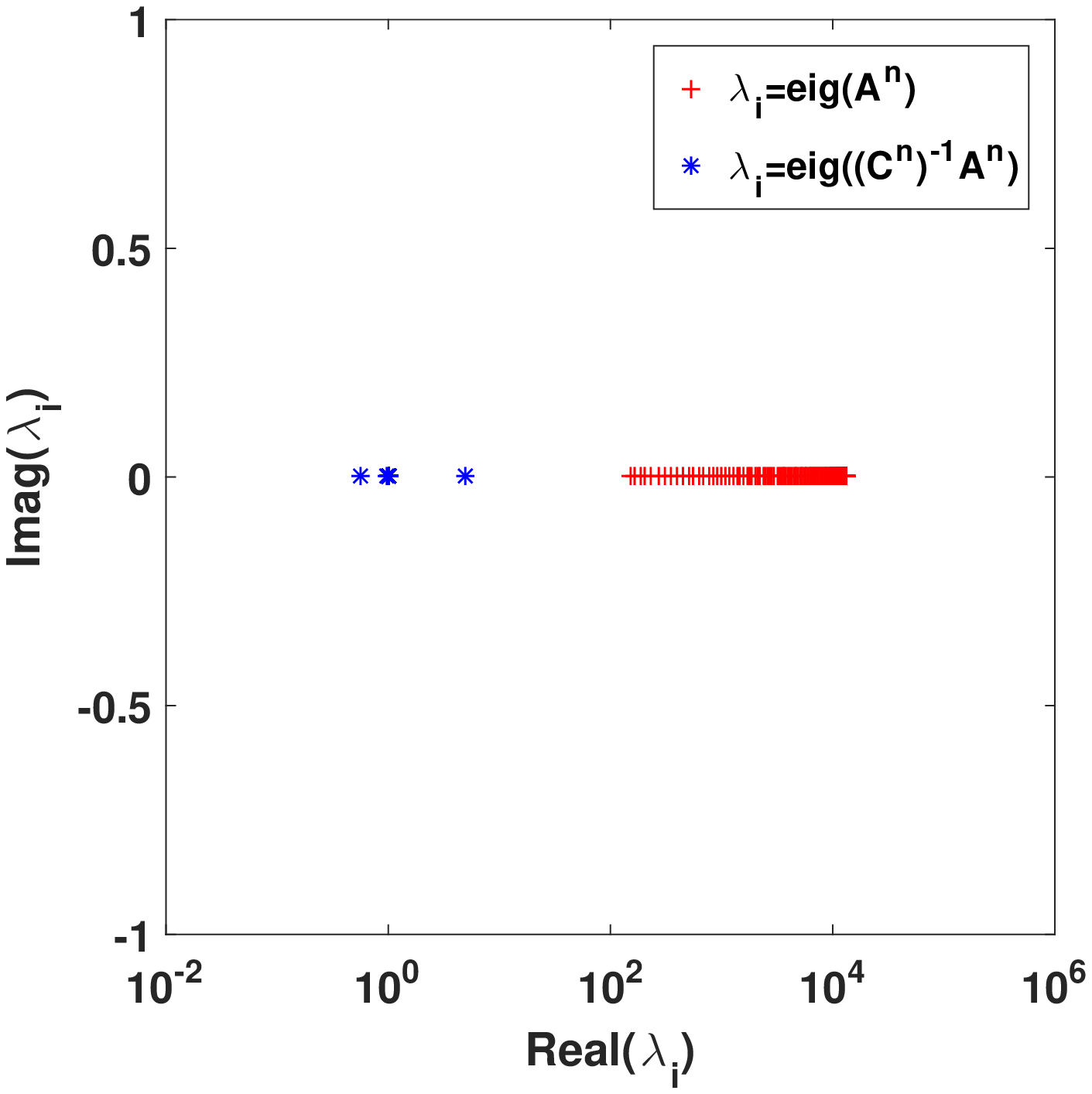}}\\
    (\textbf{a}) $n$=0 \hspace{5.8cm} (\textbf{b}) $n$ = 1 \\
  \caption{Spectrum of original matrice (red) and R. Chan-based preconditioned matrice (blue) for Example \ref{example1} at time level (\textbf{a}) $n=0$ and (\textbf{b}) $n=1$, respectively, when $M$ = $N$ = 128, $J$ = 50, $\beta$ = 1.8, and $T$ = 1.5.}
  \label{fig2}
\end{figure}

\begin{figure}[!hbt]
  \centering
  \subfigure{
    \includegraphics[width = 6.7 cm]{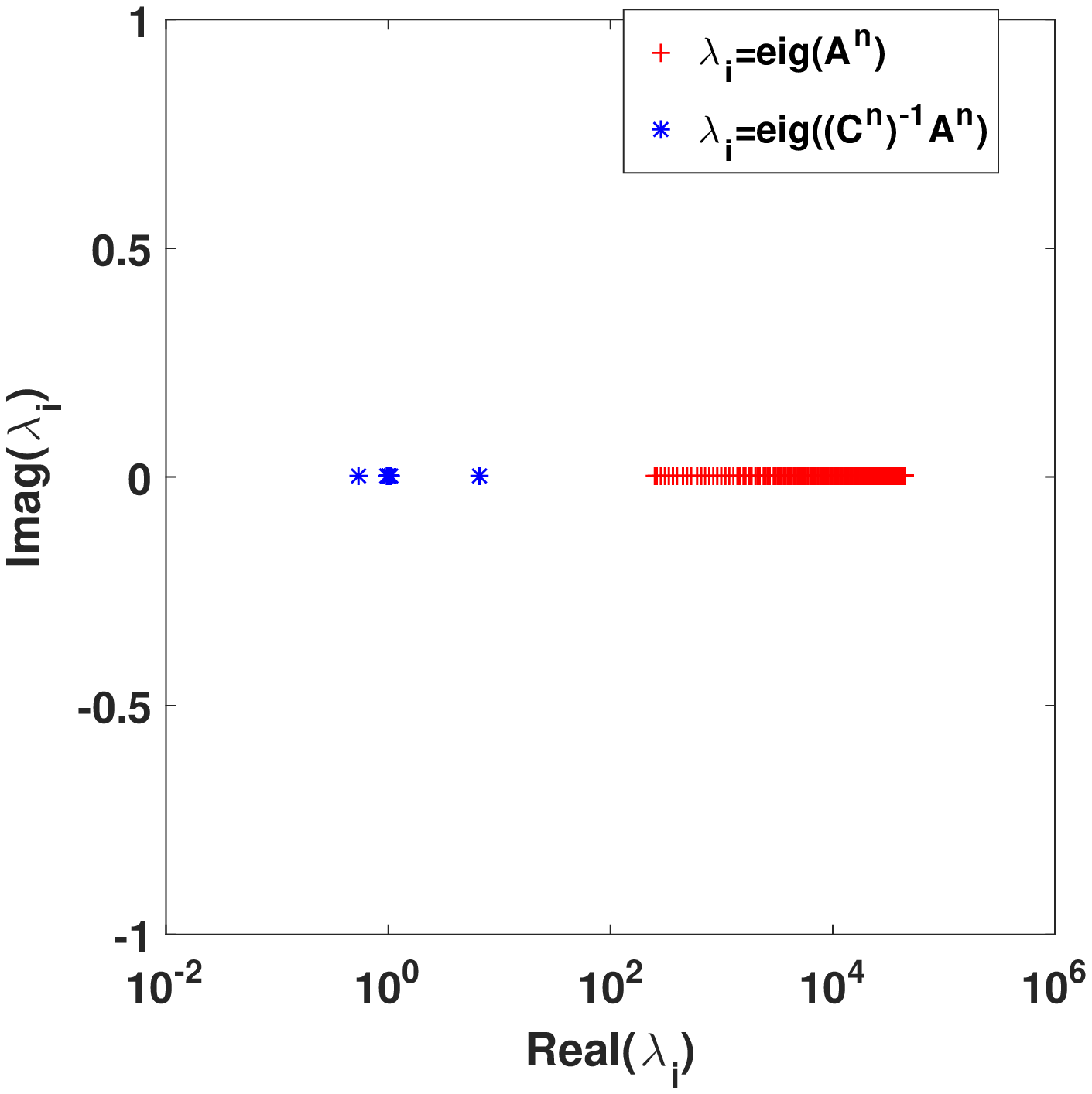}}
     \hspace{1 cm}
  \subfigure{
    \includegraphics[width = 6.7 cm]{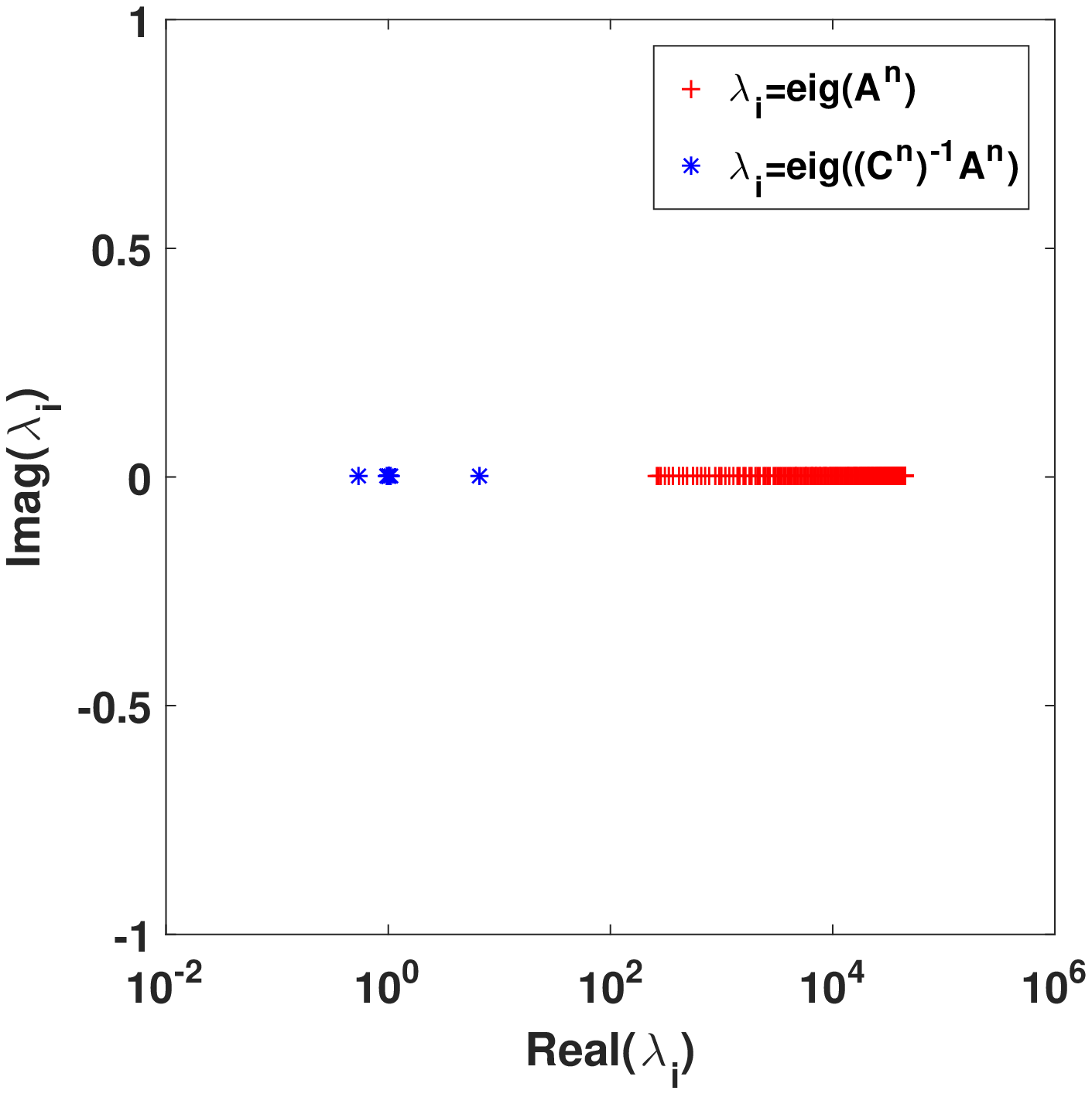}}\\
    (\textbf{a}) $n$=0 \hspace{5.8cm} (\textbf{b}) $n$ = 1 \\
  \caption{Spectrum of original matrice (red) and R. Chan-based preconditioned matrice (blue) for Example \ref{example1} at time level (\textbf{a}) $n=0$ and (\textbf{b}) $n=1$, respectively, when $M$ = $N$ = 256, $J$ = 50, $\beta$ = 1.8, and $T$ = 1.5.}
  \label{fig3}
\end{figure}

Some numerical results of the maximum errors as well as the spatial convergence orders (accuracy) for Example \ref{example1} with $\beta=1.2,~1.5$ and $1.8$ when $T=1.5,~J=50,~N=1000$ are recorded in Table \ref{tab1}. The second-order convergence of the difference scheme \eqref{2.17}-\eqref{2.19} in space can be obtained, and the results are in good agreement with what we expect.

When $T=1.5,~J=50,~M=1000$, Table \ref{tab2} provides some numerical results of the maximum errors and the temporal convergence orders for Example \ref{example1} with $\beta=1.2,~1.5$ and $1.8$. Form Table \ref{tab2}, we can see that the temporal convergence order of the difference scheme \eqref{2.17}-\eqref{2.19} is 2, which is consistent with the theoretical analysis.

Table \ref{tab3} gives the maximum errors and distributed-order integral convergence rate for Example \ref{example1} with $\beta=1.2,~1.5$ and $1.8$ respectively at $T=1.5,~M=2000,~N=2000$ and various values of $J$. The desirable second-order convergence of the difference scheme \eqref{2.17}-\eqref{2.19} in distributed-order is verified. According to the results listed in these three tables, the convergence accuracy of the difference scheme \eqref{2.17}-\eqref{2.19} of $\mathcal{O}(h^2+\tau^2+\Delta\alpha^2)$ can be observed.

From Table \ref{tab4}, one can see that the CPU time of the PCG method with circulant preconditioners is much less than that of the Cholesky method and the CG method. We also see that the PCG methods exhibit excellent performance in terms of iteration steps, and the number of iteration steps barely increases as $M$ and $N$ increase rapidly. The performance of the R. Chan-based circulant preconditioner is best among all.

The eigenvalues of the original matrix $A^{n}$ and the preconditioned matrix $(C^n)^{-1}A^n$ are plotted in Figs. \ref{fig2}-\ref{fig3}. We can see that the eigenvalues of the preconditioned matrix $(C^n)^{-1}A^n$ lie within a small interval around 1, expect for few outliers, yet all the eigenvalues are well separated away from 0. This confirms that the circulant preconditioning have nice clustering properties.


\begin{figure}[!hbt]
  \centering
  \subfigure{
    \includegraphics[width = 6 cm]{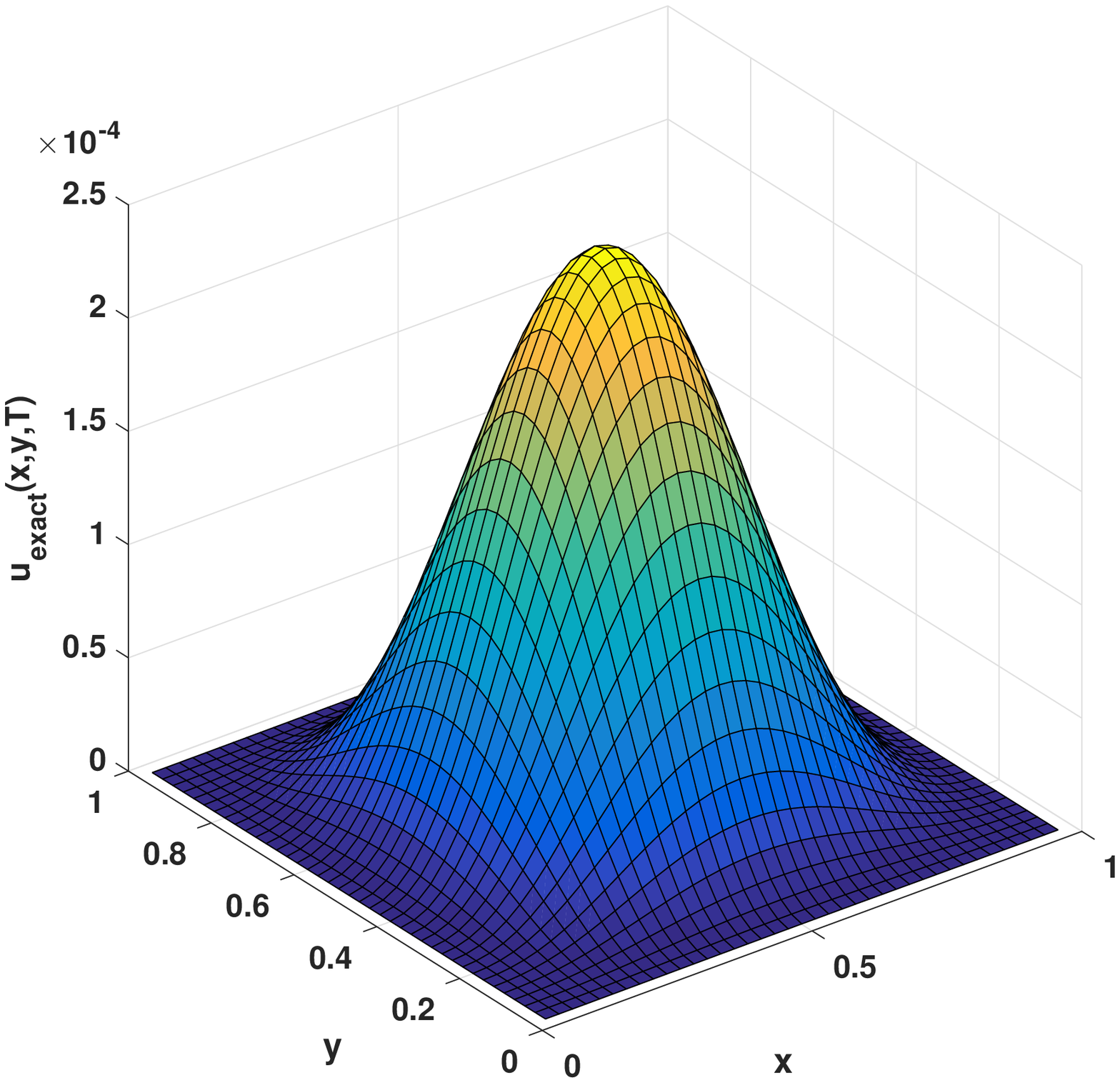}}
     \hspace{1 cm}
  \subfigure{
    \includegraphics[width = 6 cm]{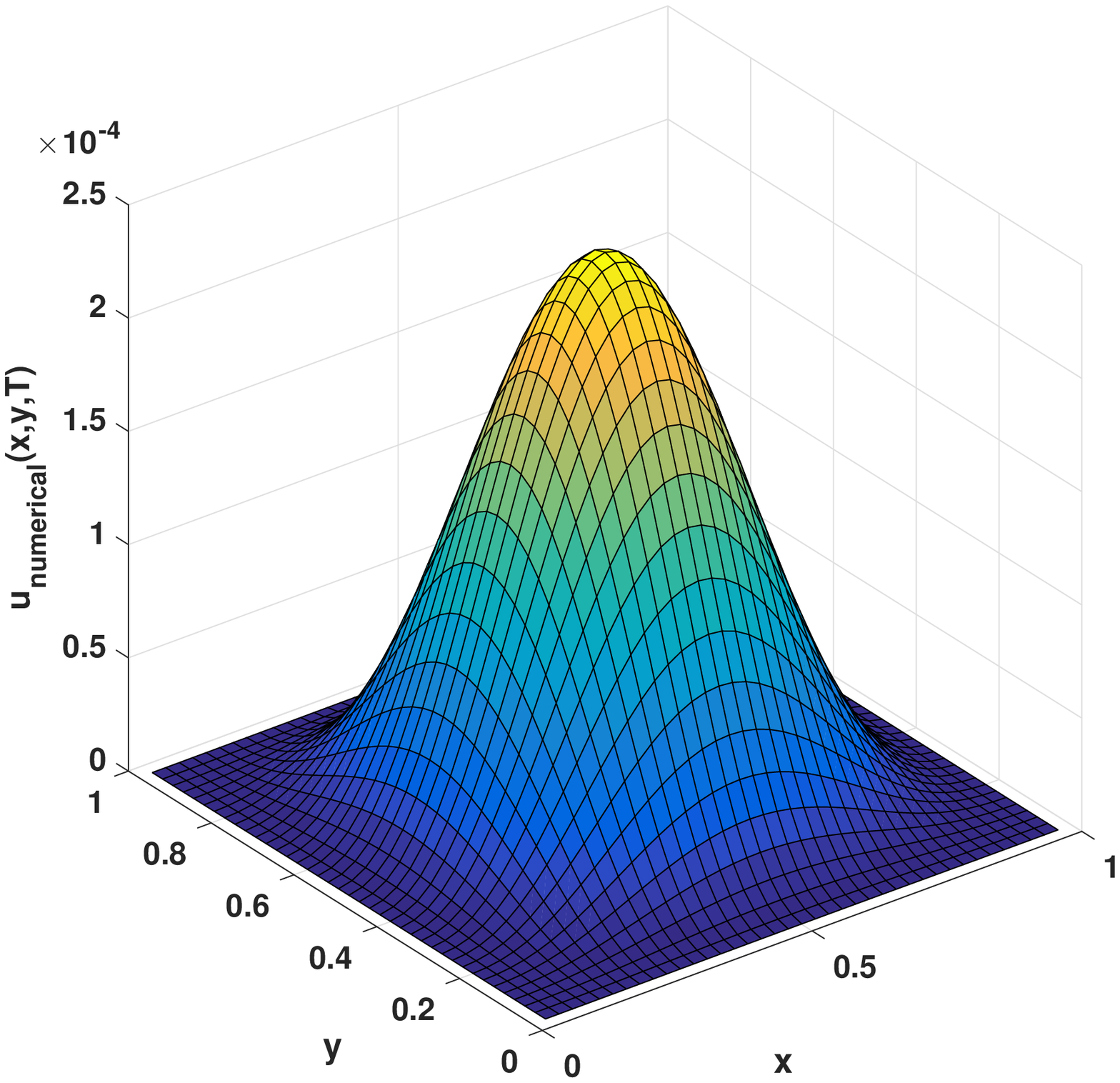}}\\
    (\textbf{a}) $T=1$, $\beta$ = $\gamma$ = 1.8 \hspace{3cm} (\textbf{b}) $T=1$, $\beta$ = $\gamma$ = 1.8 \\
     \subfigure{
    \includegraphics[width = 6 cm]{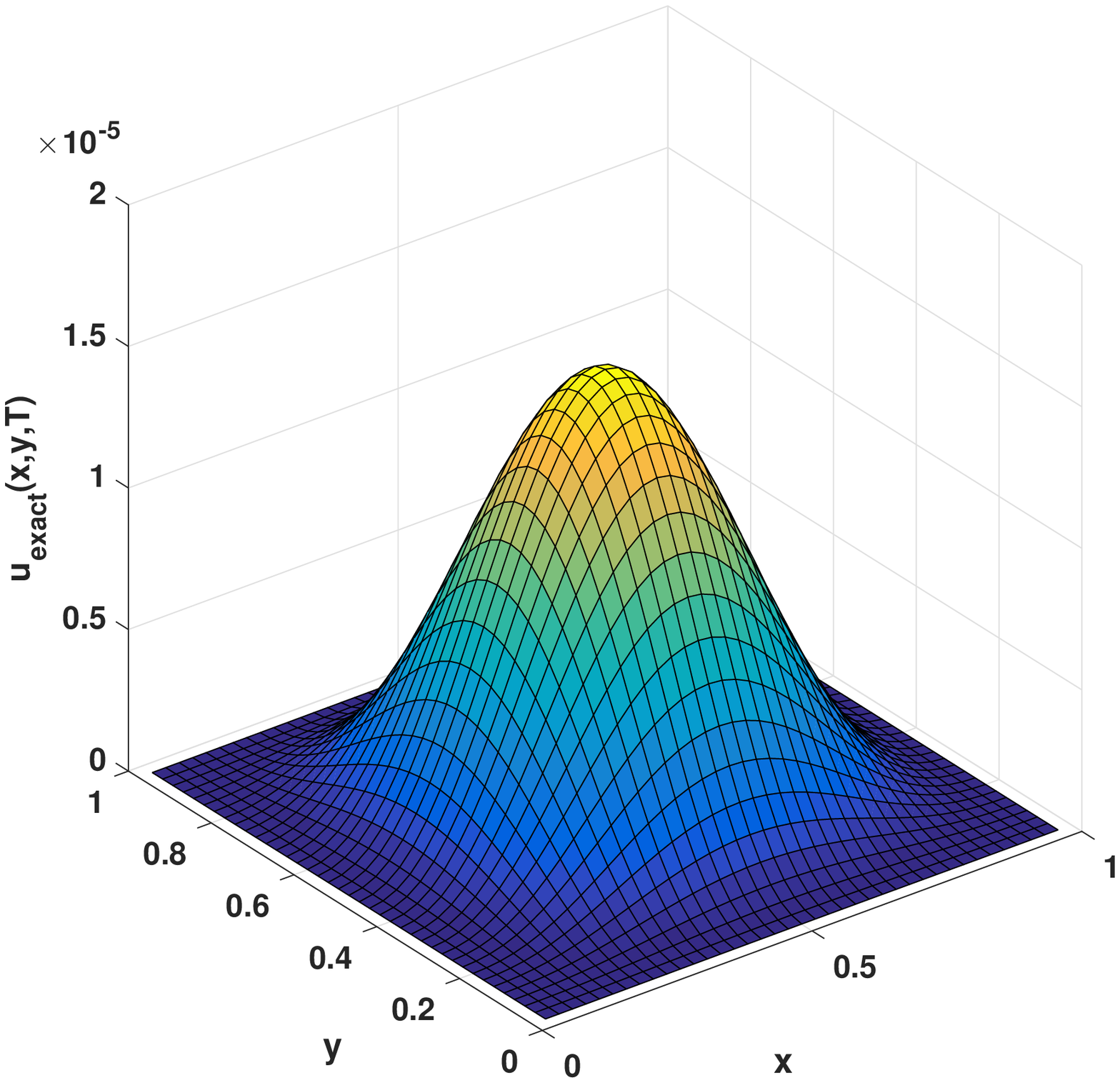}}
     \hspace{1 cm}
  \subfigure{
    \includegraphics[width = 6 cm]{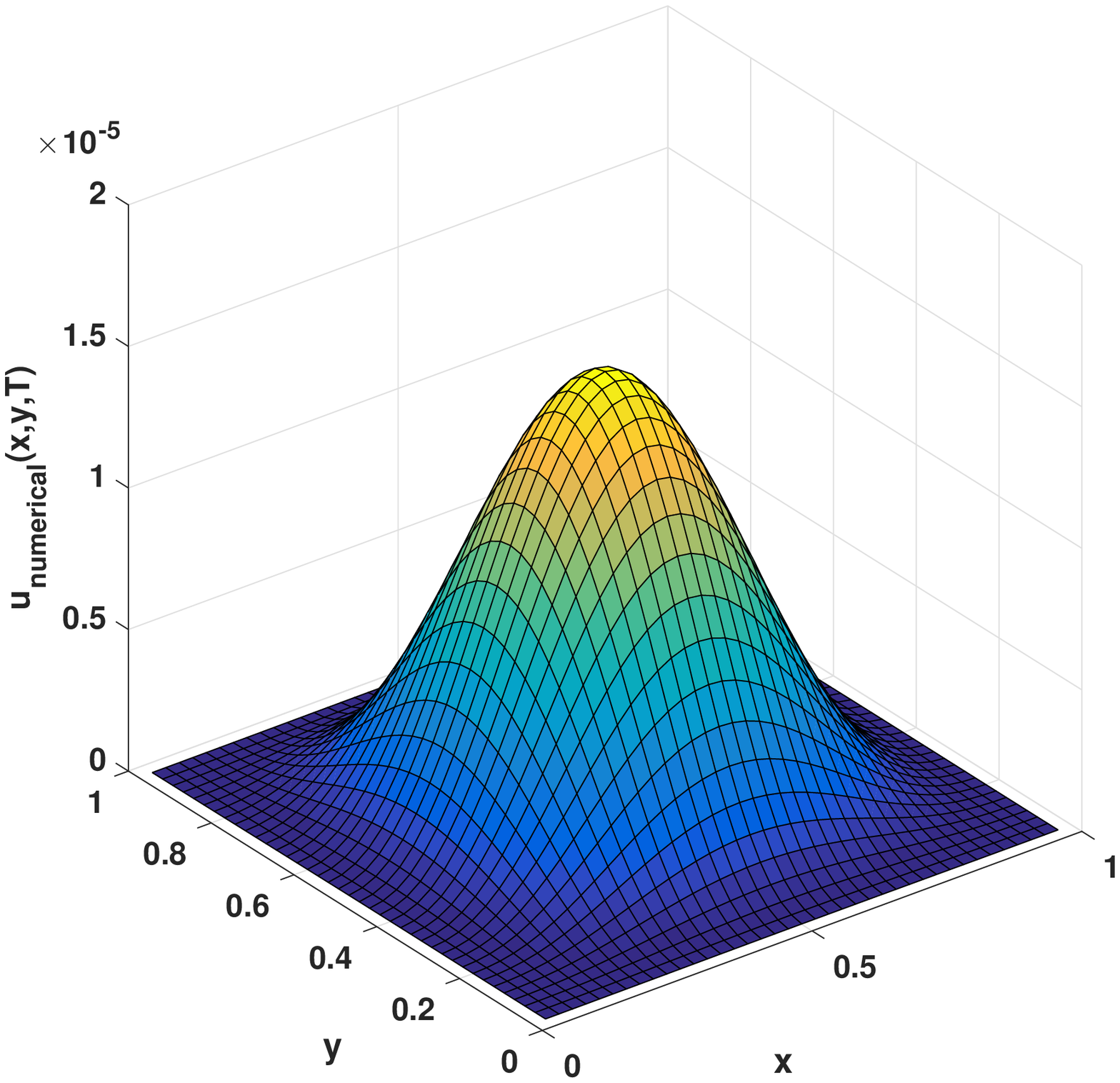}}\\
    (\textbf{c}) $T=0.5$, $\beta$ = $\gamma$ = 1.3 \hspace{3cm} (\textbf{d}) $T=0.5$, $\beta$ = $\gamma$ = 1.3 \\
  \caption{The solution surfaces obtained from Example \ref{example2} at $J$ = 50, $\widetilde{M}$ =40 and $N$ = 10: (\textbf{a}) the exact solution with $T=1$, $\beta$ = $\gamma$ = 1.8; (\textbf{b}) the numerical solution with $T=1$, $\beta$ = $\gamma$ = 1.8 by the scheme \eqref{3.12}-\eqref{3.14}; (\textbf{c}) the exact solution with $T=0.5$, $\beta$ = $\gamma$ = 1.3; (\textbf{d}) the numerical solution with $T=0.5$, $\beta$ = $\gamma$ = 1.3 by the scheme \eqref{3.12}-\eqref{3.14};}
  \label{fig4}
\end{figure}

\begin{table}[t]\small\tabcolsep=5.3pt
\begin{center}
\caption{{\small {Maximum errors and spatial convergence orders of difference scheme \eqref{3.12}-\eqref{3.14} for Example \ref{example2} with $T$ = 1.5; $J$ = 50; $N$ = 2000.}}}
\begin{tabular}{ccccccc}
\hline  & \multicolumn{2}{c}{$\beta=\gamma=1.2$} & \multicolumn{2}
{c}{$\beta=\gamma=1.5$} & \multicolumn{2}{c}{$\beta=\gamma=1.8$} \\
[-2pt] \cmidrule(lr){2-3} \cmidrule(lr){4-5} \cmidrule(lr){6-7} \\ [-11pt]
 $\widetilde{M}$ & $e({\widetilde{h}},\tau,\Delta\alpha)$ & ${\widetilde{rate}_h}$ & $e({\widetilde{h}},\tau,\Delta\alpha)$ &${\widetilde{rate}_h}$ & $e({\widetilde{h}},\tau,\Delta\alpha)$ & ${\widetilde{rate}_h}$\\
\hline
8  & 1.287397e-05 &  -    & 2.196133e-05 &   -    & 3.333244e-05 &   -    \\
16 & 3.193383e-06 &2.0113 & 5.413665e-06 & 2.0203 & 8.121444e-06 & 2.0371 \\
32 & 7.961929e-07 &2.0039 & 1.348003e-06 & 2.0058 & 2.016181e-06 & 2.0101 \\
64 & 1.982329e-07 &2.0059 & 3.360192e-07 & 2.0042 & 5.025245e-07 & 2.0044 \\
128& 4.880024e-08 &2.0222 & 8.327729e-08 & 2.0125 & 1.249136e-07 & 2.0083 \\
\hline
\end{tabular}
\label{tab5}
\end{center}
\end{table}

\begin{table}[t]\small\tabcolsep=5.3pt
\begin{center}
\caption{{\small {Maximum errors and temporal convergence orders of difference scheme \eqref{3.12}-\eqref{3.14} for Example \ref{example2} with $T$ = 1.5; $J$ = 50; $\widetilde{M}$ = 300.}}}
\begin{tabular}{ccccccc}
\hline  & \multicolumn{2}{c}{$\beta=\gamma=1.2$} & \multicolumn{2}
{c}{$\beta=\gamma=1.5$} & \multicolumn{2}{c}{$\beta=\gamma=1.8$} \\
[-2pt] \cmidrule(lr){2-3} \cmidrule(lr){4-5} \cmidrule(lr){6-7} \\ [-11pt]
 $N$  & $e({\widetilde{h}},\tau,\Delta\alpha)$ & ${\widetilde{rate}_\tau}$ & $e({\widetilde{h}},\tau,\Delta\alpha)$ &${\widetilde{rate}_\tau}$ & $e({\widetilde{h}},\tau,\Delta\alpha)$ & ${\widetilde{rate}_\tau}$\\
\hline
4 & 4.143174e-05 &   -  & 4.904054e-05&   -  & 5.654620e-05&   -   \\
8 & 1.114619e-05 &1.8942& 1.334121e-05&1.8781& 1.550841e-05&1.8664 \\
16& 2.888455e-06 &1.9482& 3.479612e-06&1.9389& 4.067734e-06&1.9308 \\
32& 7.263603e-07 &1.9915& 8.810213e-07&1.9817& 1.031666e-06&1.9792 \\
64& 1.767839e-07 &2.0387& 2.124202e-07&2.0523& 2.460280e-07&2.0681 \\
\hline
\end{tabular}
\label{tab6}
\end{center}
\end{table}

\begin{table}[t]\small\tabcolsep=5.3pt
\begin{center}
\caption{{\small {Maximum errors and distributed-order integral convergence orders of difference scheme \eqref{3.12}-\eqref{3.14} for Example \ref{example2} with $T$ = 1.5; $\widetilde{M}$ = 800; $N$ = 2000.}}}
\begin{tabular}{ccccccc}
\hline  & \multicolumn{2}{c}{$\beta=\gamma=1.2$} & \multicolumn{2}
{c}{$\beta=\gamma=1.5$} & \multicolumn{2}{c}{$\beta=\gamma=1.8$} \\
[-2pt] \cmidrule(lr){2-3} \cmidrule(lr){4-5} \cmidrule(lr){6-7} \\ [-11pt]
 $J$  & $e(\widetilde{h},\tau,\Delta\alpha)$ & $\widetilde{rate}_{\Delta\alpha}$ & $e(\widetilde{h},\tau,\Delta\alpha)$ &$\widetilde{rate}_{\Delta\alpha}$ & $e(\widetilde{h},\tau,\Delta\alpha)$ & $\widetilde{rate}_{\Delta\alpha}$\\
\hline
1 & 2.037150e-06 &   -  & 1.777609e-06&   -  & 1.489954e-06&   -   \\
2 & 5.120092e-07 &1.9923& 4.455139e-07&1.9964& 3.719972e-07&2.0019 \\
4 & 1.274115e-07 &2.0067& 1.101005e-07&2.0166& 9.090557e-08&2.0329 \\
8 & 3.105365e-08 &2.0367& 2.609633e-08&2.0769& 2.053103e-08&2.1466 \\
\hline
\end{tabular}
\label{tab7}
\end{center}
\end{table}

\begin{example}\label{example2}
Consider the following two-dimensional time distributed-order and Riesz space fractional diffusion problem:
$$\begin{cases}
\int_0^1\Gamma(5-\alpha){}^C_0D_t^{\alpha}u(x,y,t)d\alpha=\frac{\partial^{\beta}u(x,y,t)}{\partial\mid
x\mid^{\beta}}+\frac{\partial^{\gamma}u(x,y,t)}{\partial\mid
y\mid^{\gamma}}+f(x,y,t),\quad (x,y)\in\Omega,~0<t\leq T,\\
u(x,y,t)=0,\quad (x,y)\in\partial\Omega,~0\leq t\leq T,\\
u(x,y,0)=0,\quad (x,y)\in\Omega,
\end{cases}$$
with $\Omega=(0,1)\times(0,1)$ and
\begin{align*}
f(x,y,t)=&f_0(x,y,t)-c_1t^4y^3(1-y)^3[f_1(x,y,t)-3f_2(x,y,t)+3f_3(x,y,t)-f_4(x,y,t)]\\
&-c_2t^4x^3(1-x)^3[g_1(x,y,t)-3g_2(x,y,t)+3g_3(x,y,t)-g_4(x,y,t)],
\end{align*}
where
\begin{align*}
c_1=-\frac{1}{2\cos(\beta\pi/2)},\quad c_2=-\frac{1}{2\cos(\gamma\pi/2)},
\end{align*}
and
\begin{align*}
&f_0(x,y,t)=24t^3(t-1)x^3(1-x)^3y^3(1-y)^3/\ln t,\\
&f_1(x,y,t)={\Gamma(4)}/{\Gamma(4-\beta)}[x^{3-\beta}+(1-x)^{3-\beta}],\\
&f_2(x,y,t)={\Gamma(5)}/{\Gamma(5-\beta)}[x^{4-\beta}+(1-x)^{4-\beta}],\\
&f_3(x,y,t)={\Gamma(6)}/{\Gamma(6-\beta)}[x^{5-\beta}+(1-x)^{5-\beta}],\\
&f_4(x,y,t)={\Gamma(7)}/{\Gamma(7-\beta)}[x^{6-\beta}+(1-x)^{6-\beta}],\\
&g_1(x,y,t)={\Gamma(4)}/{\Gamma(4-\gamma)}[y^{3-\gamma}+(1-y)^{3-\gamma}],\\
&g_2(x,y,t)={\Gamma(5)}/{\Gamma(5-\gamma)}[y^{4-\gamma}+(1-y)^{4-\gamma}],\\
&g_3(x,y,t)={\Gamma(6)}/{\Gamma(6-\gamma)}[y^{5-\gamma}+(1-y)^{5-\gamma}],\\
&g_4(x,y,t)={\Gamma(7)}/{\Gamma(7-\gamma)}[y^{6-\gamma}+(1-y)^{6-\gamma}].\\
\end{align*}
The exact solution of the example is $u(x,t)=t^4x^3(1-x)^3y^3(1-y)^3$.
\end{example}

\begin{table}[t]\small\tabcolsep=6.0pt
\begin{center}
\caption{{\small {Comparisons on Example \ref{example2} between the Cholesky method, the CG method, and the PCG method with different circulant preconditioners, where $\beta$ = $\gamma$ = 1.2, 1.5, 1.8, $J$ = 50 and $T$ = 1.5.}}}
\begin{tabular}{cccccccccccc}
\\
\hline & & &\multicolumn{1}{c}{$\rm{Chol}$} & \multicolumn{2}
{c}{$\rm{CG}$}& \multicolumn{2}{c}{$\rm{PCG(S)}$} & \multicolumn{2}
{c}{$\rm{PCG(T)}$} & \multicolumn{2}{c}{$\rm{PCG(C)}$} \\
[-2pt] \cmidrule(lr){4-4} \cmidrule(lr){5-6} \cmidrule(lr){7-8} \cmidrule(lr){9-10} \cmidrule(lr){11-12} \\ [-11pt]
 $\beta$=$\gamma$  & $\widetilde{M}$ & $N$ & $\rm{CPU(s)}$ & $\rm{CPU(s)}$ & $\rm{Iter}$ & $\rm{CPU(s)}$ & $\rm{Iter}$ & $\rm{CPU(s)}$ & $\rm{Iter}$ & $\rm{CPU(s)}$ & $\rm{Iter}$\\
\hline
    &$2^3$ &$2^2$ &0.00  &0.00 &10.0 &0.01 &11.0 &0.01 &8.0  &\textbf{0.01} &\textbf{8.0} \\
    &$2^4$ &$2^3$ &0.01  &0.02 &17.0 &0.02 &11.0 &0.02 &9.0  &\textbf{0.02} &\textbf{9.0} \\
1.2 &$2^5$ &$2^4$ &0.19  &0.24 &23.0 &0.17 &11.0 &0.16 &10.0 &\textbf{0.16} &\textbf{10.0} \\
    &$2^6$ &$2^5$ &7.17  &1.17 &29.0 &0.80 &11.0 &0.77 &10.2 &\textbf{0.80} &\textbf{11.0} \\
    &$2^7$ &$2^6$ &606.65&14.87&35.0 &7.55 &11.0 &7.67 &11.0 &\textbf{7.74} &\textbf{11.0} \\
  \\
    &$2^3$ &$2^2$ &0.00  &0.00 &10.0 &0.01 &14.0 &0.01 &8.3  &\textbf{0.00} &\textbf{8.0} \\
    &$2^4$ &$2^3$ &0.01  &0.03 &22.0 &0.02 &13.0 &0.02 &11.0 &\textbf{0.02} &\textbf{10.0} \\
1.5 &$2^5$ &$2^4$ &0.18  &0.33 &35.0 &0.21 &15.0 &0.19 &13.0 &\textbf{0.18} &\textbf{12.0} \\
    &$2^6$ &$2^5$ &7.17  &1.89 &54.0 &0.93 &14.1 &0.98 &15.0 &\textbf{0.88} &\textbf{13.0} \\
    &$2^7$ &$2^6$ &605.57&29.15&79.0 &9.64 &16.0 &10.34&17.0 &\textbf{9.49} &\textbf{15.0} \\
 \\
    &$2^3$ &$2^2$ &0.00  &0.01 &10.0 &0.01 &15.3 &0.01 &10.0 &\textbf{0.00} &\textbf{8.0}  \\
    &$2^4$ &$2^3$ &0.02  &0.04 &27.0 &0.03 &16.0 &0.03 &13.0 &\textbf{0.02} &\textbf{11.0} \\
1.8 &$2^5$ &$2^4$ &0.18  &0.44 &48.9 &0.24 &18.0 &0.23 &17.0 &\textbf{0.19} &\textbf{13.0} \\
    &$2^6$ &$2^5$ &7.12  &2.72 &87.9 &1.13 &19.0 &1.21 &20.9 &\textbf{0.96} &\textbf{15.0} \\
    &$2^7$ &$2^6$ &606.44&51.03&150.0&11.83&21.0 &14.44&26.3 &\textbf{10.84}&\textbf{18.0} \\
\hline
\end{tabular}
\label{tab8}
\end{center}
\end{table}
For simplicity, take $h_1=h_2=\widetilde{h}$, and $M_1=M_2=\widetilde{M}$.
Let $e(\widetilde{h},\tau,\Delta\alpha)=\max\limits_{{0\leq i\leq M_1},~{0\leq j\leq M_2}\atop{0\leq n\leq N}}\\
|u(x_i,y_j,t_n, \Delta\alpha)-u_{ij}^n|$, where $u(x_i,y_j,t_n,\Delta\alpha)$ and $u_{ij}^n$ represent the exact solution and numerical solution with the step sizes $\widetilde{h}$, $\tau$ and $\Delta\alpha$, respectively. The convergence orders are defined as
$$\widetilde{rate}_h=\log_2\frac{e(\widetilde{h},\tau,\Delta\alpha)}{e(\widetilde{h}/2,\tau,\Delta\alpha)},~
\widetilde{rate}_{\tau}=\log_2\frac{e(\widetilde{h},\tau,\Delta\alpha)}{e(\widetilde{h},\tau/2,\Delta\alpha)},~
\widetilde{rate}_{\Delta\alpha}=\log_2\frac{e(\widetilde{h},\tau,\Delta\alpha)}{e(\widetilde{h},\tau,{\Delta\alpha}/2)}.$$

Fig. \ref{fig4} exhibits the solution surface of Example \ref{example2} with $J$ = 50, $\widetilde{M}$ =40, $N$ = 10 at $T=1$, $\beta$ = $\gamma$ = 1.8 and $T=0.5$, $\beta$ = $\gamma$ = 1.3, respectively. It can be seen that the numerical solutions are in good conformity with the exact solutions.

We compute the convergence orders in spatial of the difference scheme \eqref{3.12}-\eqref{3.14} for Example \ref{example2}. When $T$ = 1.5, $J$ = 50 and $N$ = 2000, Table \ref{tab5} lists the maximum errors and convergence orders in spatial of the difference scheme with $\beta$ = $\gamma$ = 1.2, 1.5 and 1.8, respectively. From the numerical results we can conclude that the difference scheme \eqref{3.12}-\eqref{3.14} has the second-order convergence in spatial directions.

When taking the fixed $T$ = 1.5, $J$ = 50, $\widetilde{M}$ = 100, the maximum errors and convergence orders in temporal of the difference scheme \eqref{3.12}-\eqref{3.14} with $\beta$ = $\gamma$ = 1.2, 1.5 and 1.8 are listed in Table \ref{tab6}, respectively. From the numerical results in Table \ref{tab6} we can clearly see that the convergence order in temporal of the difference scheme \eqref{3.12}-\eqref{3.14} is also nearly 2, which is in accord with the theoretical analysis.

The numerical accuracy of scheme \eqref{3.12}-\eqref{3.14} for Example \ref{example2} in distributed-order integral variable is investigated. When $T$ = 1.5, $J$ = 50, $\widetilde{M}$ = 100, Table \ref{tab7} displays the computational results using the difference scheme \eqref{3.12}-\eqref{3.14} with $\beta$ = $\gamma$ = 1.2, 1.5 and 1.8, respectively. One can draw the conclusion that the convergence accuracy of distributed-order integral variable is $\mathcal{O}(\Delta\alpha^2)$.  Namely, the numerical convergence order of the difference scheme \eqref{3.12}-\eqref{3.14} is $\mathcal{O}(h_1^2+h_2^2+\tau^2+\Delta\alpha^2)$.

From Table \ref{tab8}, we can observe that the CPU time of the PCG method with circulant preconditioners is much less than that of the Cholesky method and the CG method. We also see that the number of iteration steps of the PCG method barely increases as the number of the spatial grid points increases. The performance of the R. Chan-based circulant preconditioner is best amongst all.

The spectrum of the original matrix $M^{n}$ and the preconditioned matrix $(C_2^n)^{-1}M^n$ are plotted in Figs. \ref{fig5}-\ref{fig6}. These two figures also confirm that the circulant preconditioning have nice clustering properties. It shows that the eigenvalues of the preconditioned matrix are well grouped around 1 expect for few outliers. The vast majority of the eigenvalues are well separated away from 0.

\begin{figure}[!hbt]
  \centering
  \subfigure{
    \includegraphics[width = 6.7 cm]{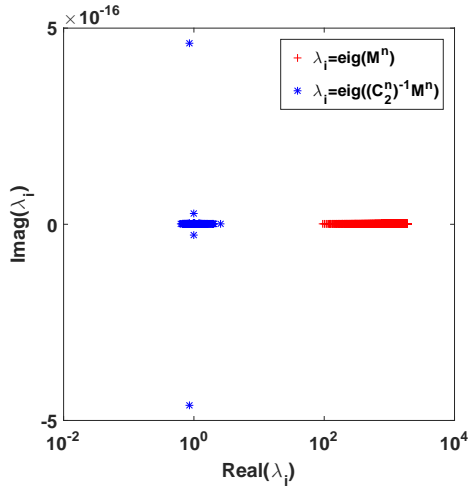}}
     \hspace{1 cm}
  \subfigure{
    \includegraphics[width = 6.7 cm]{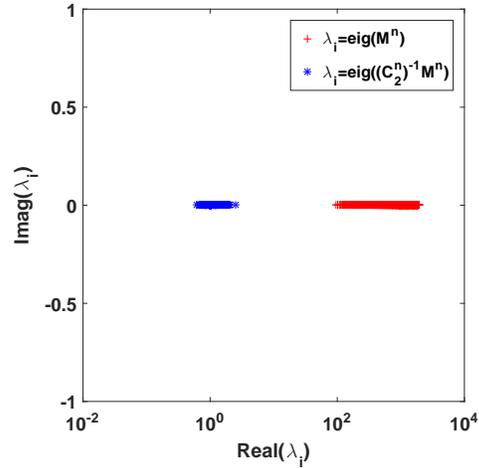}}\\
    (\textbf{a}) $n$=0 \hspace{5.8cm} (\textbf{b}) $n$ = 1 \\
  \caption{Spectrum of original matrice (red) and R. Chan-based preconditioned matrice (blue)  for Example \ref{example2} at time level (\textbf{a}) $n=0$ and (\textbf{b}) $n=1$, respectively, when $\widetilde{M}$ = $N$ = 64, $J$ = 50, $\beta$ = $\gamma$ = 1.5, and $T$ = 1.5.}
  \label{fig5}
\end{figure}

\begin{figure}[!hbt]
  \centering
  \subfigure{
    \includegraphics[width = 6.7 cm]{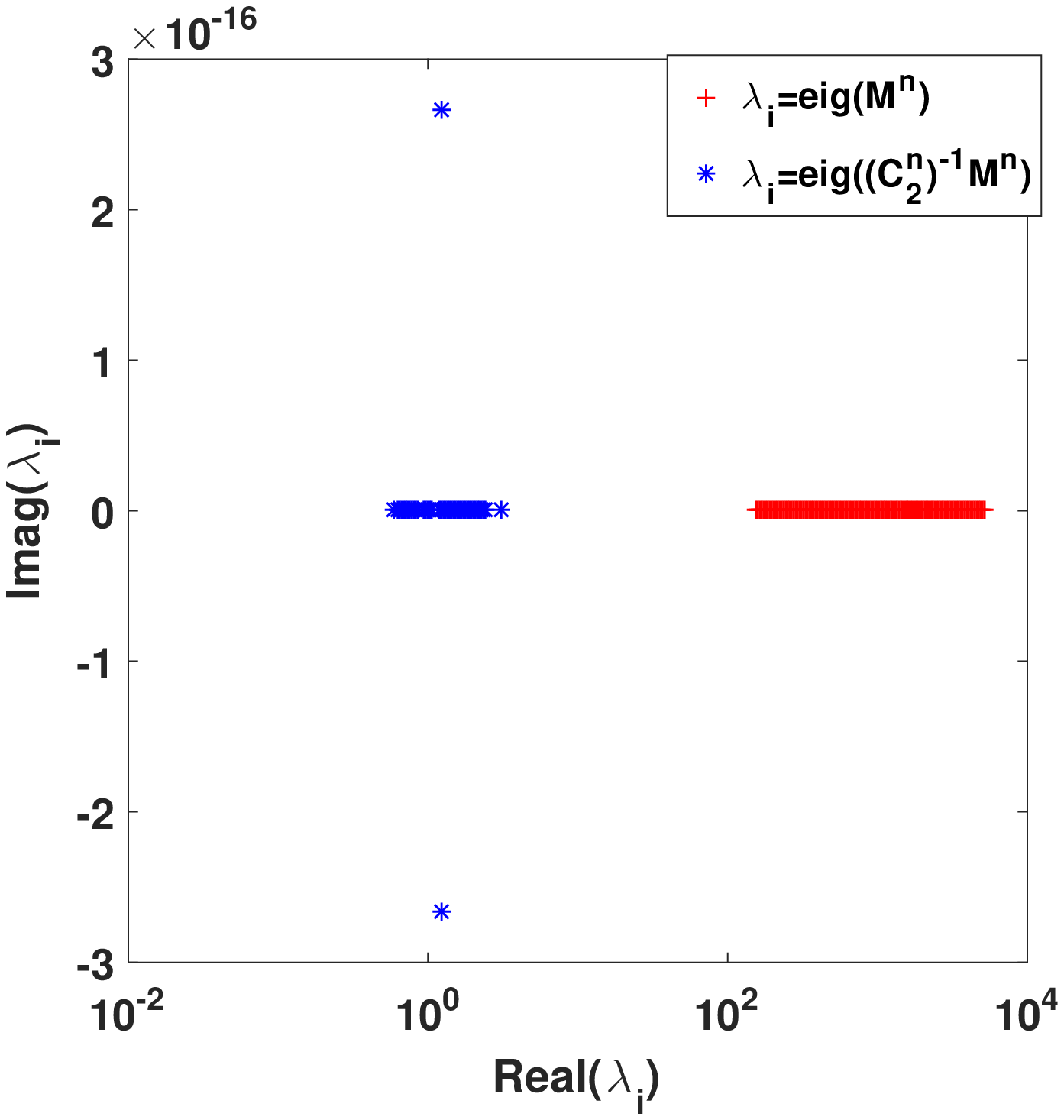}}
     \hspace{1 cm}
  \subfigure{
    \includegraphics[width = 6.7 cm]{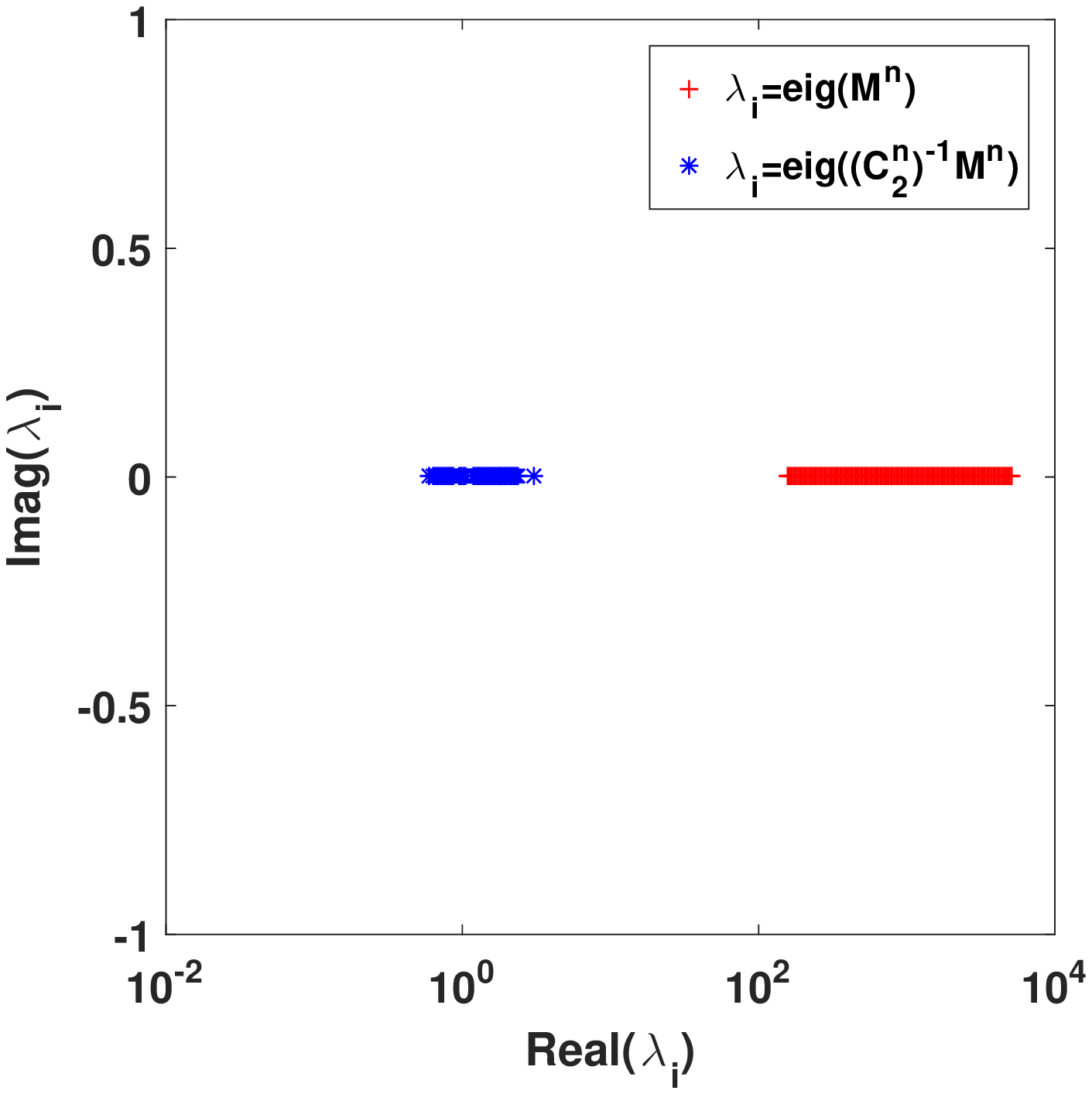}}\\
    (\textbf{a}) $n$=0 \hspace{5.8cm} (\textbf{b}) $n$ = 1 \\
  \caption{Spectrum of original matrice (red) and R. Chan-based preconditioned matrice (blue) for Example \ref{example2} at time level (\textbf{a}) $n=0$ and (\textbf{b}) $n=1$, respectively, when $\widetilde{M}$ = $N$ = 128, $J$ = 50, $\beta$ = $\gamma$ = 1.5, and $T$ = 1.5.}
  \label{fig6}
\end{figure}

\section{Conclusion}\label{section5}

\quad In this paper, several efficient second-order difference schemes are proposed for one- and two-dimensional
TDRFDEs. We first discretize the time distributed-order integral term by using composite trapezoid formula and transform the TDRFDEs into the multi-term time-space FDEs. Then we solve the multi-term time-space FDEs with the second-order accurate interpolation approximation on a special point. We prove that the proposed difference schemes are uniquely solvable, unconditionally stable and convergent in the mesh $L_2$-norm with second-order accuracy in time, space and distributed-order integral variables. Moreover, we have proposed an efficient implementation of the proposed scheme based on the PCG method with R. Chan-based circulant preconditioner, which only requires $\mathcal{O}((M-1)\log(M-1))$ computational complexity and $\mathcal{O}((M-1)$ storage cost. Numerical experiments confirm the theoretical results and show the effectiveness of the proposed preconditioned method.
In future work, we will focus on the development of the effective numerical methods for solving high-dimensional time distributed-order fractional diffusion-wave equations.

\section*{Acknowledgments}
\quad {\it This research is supported by NSFC (61772003, 11501085, 11601365 and 11701467),
the Fundamental Research Funds for the Central Universities (ZYGX2016J132 and ZYGX2016\\
J138) and the Shandong Science and Technology Department Foundation (J16LI06).}



\begin{thebibliography}{widest-label}
\bibitem{scher1975anomalous}H. Scher, E. W. Montroll, \textit{Anomalous transit-time dispersion in amorphous solids}, Phys. Rev. B {\bf 12} (6) (1975) 2455.
\bibitem{schneider1989fractional}W. Schneider, W. Wyss, \textit{Fractional diffusion and wave equations}, J. Math. Phys. {\bf 30} (1) (1989) 134-144.
\bibitem{gao2015some}G.-H. Gao, H.-W. Sun, Z.-Z. Sun, \textit{ Some high-order difference schemes for the
distributed-order differential equations}, J. Comput. Phys. {\bf 298} (2015) 337-359.
\bibitem{Zhao2013Total}X.-L. Zhao, W. Wang, T.-Y. Zeng, T.-Z. Huang, M. K. Ng, \textit{Total variation structured total least squares method for image restoration}, SIAM J. Sci. Comput. {\bf 35} (6) (2013)  B1304-B1320.
\bibitem{henry2006anomalous}B. Henry, T. Langlands, S. Wearne, \textit{Anomalous diffusion with linear reaction dynamics: from continuous time random walks to fractional reaction-diffusion equations}, Phys. Rev. E {\bf 74} (3) (2006) 031116.
\bibitem{zhuang2008new}P.-H. Zhuang, F.-W. Liu, V. Anh, I. Turner, \textit{New solution and analytical techniques of the implicit numerical method for the anomalous subdiffusion equation}, SIAM J. Numer. Anal. {\bf 46} (2) (2008) 1079-1095.
\bibitem{sun2009variable}H.-G. Sun, W. Chen, Y.-Q. Chen, \textit{Variable-order fractional differential operators in anomalous diffusion modeling}, Phys. A {\bf 388} (21) (2009) 4586-4592.
\bibitem{magin2008anomalous}R. L. Magin, O. Abdullah, D. Baleanu, X. J. Zhou, \textit{Anomalous diffusion expressed through fractional order differential operators in the Bloch-Torrey equation}, J. Magn. Reson. {\bf 190} (2) (2008) 255-270.
\bibitem{cui2009compact}M. Cui, \textit{Compact finite difference method for the fractional diffusion equation}, J. Comput. Phys. {\bf 228} (20) (2009) 7792-7804.
\bibitem{li2014higher}C. Li, H. Ding, \textit{Higher order finite difference method for the reaction and anomalous-diffusion equation}, Appl. Math. Model. {\bf 38} (15) (2014) 3802-3821.
\bibitem{langlands2006solution}T. A. M. Langlands, \textit{Solution of a modified fractional diffusion equation}, Phys. A {\bf 367} (2006) 136-144.
\bibitem{liu2009numerical}F. Liu, C. Yang, K. Burrage, \textit{Numerical method and analytical technique of the modified anomalous subdiffusion equation with a nonlinear source term}, J. Comput. Appl. Math. {\bf 231} (1) (2009) 160-176.
\bibitem{liu2011finite}Q. Liu, F. Liu, I. Turner, V. Anh, \textit{Finite element approximation for a modified anomalous subdiffusion equation}, Appl. Math. Model. {\bf 35} (8) (2011) 4103-4116.
\bibitem{zheng2016high}M. Zheng, F. Liu, V. Anh, I. Turner, \textit{A high-order spectral method for the multi-term time-fractional diffusion equations}, Appl. Math. Model. {\bf 40} (7) (2016) 4970-4985.
\bibitem{gu2015strang}X.-M. Gu, T.-Z. Huang, X.-L. Zhao, H.-B. Li, L. Li, \textit{Strang-type preconditioners for solving fractional diffusion equations by boundary value methods}, J. Comput. Appl. Math. {\bf 277} (2015) 73-86.
\bibitem{luo2016quadratic}W.-H. Luo, T.-Z. Huang, G.-C. Wu, X.-M. Gu, \textit{Quadratic spline collocation method for the time fractional subdiffusion equation}, Appl. Math. Comput. {\bf 276} (2016) 252-265.
\bibitem{Li2017meng}M. Li, C. Huang, F. Jiang, \textit{Galerkin finite element method for higher dimensional multi-term fractional diffusion equation on non-uniform meshes}, Appl. Anal. {\bf 96} (8) (2017) 1269-1284.
\bibitem{chen2007fourier}C.-M. Chen, F. Liu, I. Turner, V. Anh, \textit{A Fourier method for the fractional diffusion equation describing sub-diffusion}, J. Comput. Phys. {\bf 227} (2) (2007) 886-897.
\bibitem{lin2007finite}Y. Lin, C. Xu, \textit{Finite difference/spectral approximations for the time-fractional diffusion equation}, J. Comput. Phys. {\bf 225} (2) (2007) 1533-1552.
\bibitem{gu2016fast}X.-M. Gu, T.-Z. Huang, C.-C. Ji, B. Carpentieri, A. A. Alikhanov, \textit{Fast iterative method with a second order implicit difference scheme for time-space fractional convection-diffusion equations}, J. Sci. Comput. {\bf 72} (2017) 957-985.
\bibitem{chen2010numerical}C.-M. Chen, F. Liu, I. Turner, V. Anh, \textit{Numerical schemes and multivariate extrapolation of a two-dimensional anomalous sub-diffusion equation}, Numer. Algor. {\bf 54} (1) (2010) 1-21.
\bibitem{cui2013convergence}M. Cui, \textit{Convergence analysis of high-order compact alternating direction implicit schemes for the two-dimensional time fractional diffusion equation}, Numer. Algor. {\bf 62} (3) (2013) 383-409.
\bibitem{zhang2014error}Y.-N. Zhang, Z.-Z. Sun, \textit{Error analysis of a compact ADI scheme for the 2D fractional subdiffusion equation}, J. Sci. Comput. {\bf 59} (1) (2014) 104-128.
\bibitem{chechkin2002retarding}A. Chechkin, R. Gorenflo, I. Sokolov, \textit{Retarding subdiffusion and accelerating superdiffusion governed by distributed-order fractional diffusion equations}, Phys. Rev. E {\bf 66} (4) (2002) 046129.
\bibitem{kochubei2008distributed}A. N. Kochubei, \textit{Distributed order calculus and equations of ultraslow diffusion}, J. Math. Anal. Appl. {\bf 340} (1) (2008) 252-281.
\bibitem{gao2015two}G.-H. Gao, Z.-Z. Sun, \textit{Two alternating direction implicit difference schemes with the extrapolation method for the two-dimensional distributed-order differential equations}, Comput. Math. Appl. {\bf 69} (9) (2015) 926-948.
\bibitem{hu2016implicit}X. Hu, F. Liu, I. Turner, V. Anh, \textit{An implicit numerical method of a new time distributed-order and two-sided space-fractional advection-dispersion equation}, Numer. Algor. {\bf 72} (2) (2016) 393-407.
\bibitem{katsikadelis2014numerical}J. T. Katsikadelis, \textit{Numerical solution of distributed order fractional differential equations}, J. Comput. Phys. {\bf 259} (2014) 11-22.
\bibitem{ye2014numerical}H. Ye, F. Liu, V. Anh, I. Turner, \textit{Numerical analysis for the time distributed-order and Riesz space fractional diffusions on bounded domains}, IMA J. Appl. Math. {\bf 80} (3) (2015) 825-838.
\bibitem{bu2017finite}W. Bu, A. Xiao, W. Zeng, \textit{Finite difference/finite element methods for distributed-order time fractional diffusion equations}, J. Sci. Comput. {\bf 72} (3) (2017) 422-441.
\bibitem{liu2013numerical}F. Liu, M. M. Meerschaert, R. J. McGough, P. Zhuang, Q. Liu, \textit{Numerical methods for solving the multi-term time-fractional wave-diffusion equation}, Fract. Calc. Appl. Anal. {\bf 16} (1) (2013) 9-25.
\bibitem{jiang2012analytical}H. Jiang, F. Liu, I. Turner, K. Burrage, \textit{Analytical solutions for the multi-term time-space Caputo-Riesz fractional advection-diffusion equations on a finite domain}, J. Math. Anal. Appl. {\bf 389} (2) (2012) 1117-1127.
\bibitem{meerschaert2011distributed}M. M. Meerschaert, E. Nane, P. Vellaisamy, \textit{Distributed-order fractional diffusions on bounded domains}, J. Math. Anal. Appl. {\bf 379} (1) (2011) 216-228.
\bibitem{li2017analyticity}Z. Li, Y. Luchko, M. Yamamoto, \textit{Analyticity of solutions to a distributed order time-fractional diffusion equation and its application to an inverse problem}, Comput. Math. Appl. {\bf 73} (6) (2017) 1041-1052.
\bibitem{lei2013circulant}S.-L. Lei, H.-W. Sun, \textit{A circulant preconditioner for fractional diffusion equations}, J. Comput. Phys. {\bf 242} (2013) 715-725.
\bibitem{gu2015hybridized}X.-M. Gu, T.-Z. Huang, B. Carpentieri, L. Li, C. Wen, \textit{A hybridized iterative algorithm of the BiCORSTAB and GPBiCOR methods for solving non-Hermitian linear systems}, Comput. Math. Appl. {\bf 70} (12) (2015) 3019-3031.
\bibitem{gu2015k}X.-M. Gu, T.-Z. Huang, H.-B. Li, L. Li, W.-H. Luo, \textit{On k-step CSCS-based polynomial preconditioners for Toeplitz linear systems with application to fractional diffusion equations}, Appl. Math. Lett. {\bf 42} (2015) 53-58.
\bibitem{wang2011fast}K. Wang, H. Wang, \textit{A fast characteristic finite difference method for fractional advection-diffusion equations}, Adv. Water Resour. {\bf 34} (7) (2011) 810-816.
\bibitem{Zhao2012DCT}X.-L. Zhao, T.-Z. Huang, S.-L. Wu, Y.-F. Jing, \textit{DCT- and DST-based splitting methods for Toeplitz systems}, Int. J. Comput. Math. {\bf 89} (5) (2012) 691-700.
\bibitem{Yang2010}Q. Yang, F. Liu, I. Turner, \textit{Numerical methods for fractional partial differential equations with Riesz space fractional derivatives}, Appl. Math. Model. {\bf 34} (2010) 200-218.
\bibitem{chan2007introduction}R. Chan, X.-Q. Jin, \textit{An Introduction to Iterative Toeplitz Solvers}, SIAM, PA, 2007.
\bibitem{chan1989toeplitz}R. Chan, G. Strang, \textit{Toeplitz equations by conjugate gradients with circulant preconditioner}, SIAM J. Sci. Stat. Comput. {\bf 10} (1) (1989) 104-119.
\bibitem{leichenzhang2016}S.-L. Lei, X. Chen, X. Zhang, \textit{Multilevel circulant preconditioner for high-dimensional fractional diffusion equations}, East Asian J. Appl. Math. {\bf 6} (2) (2016) 109-130.
\bibitem{chou2017385}L.-K. Chou, S.-L. Lei, \textit{Fast ADI method for high dimensional fractional diffusion equations in conservative form with preconditioned strategy}, Comput. Math. Appl. {\bf 73} (3) (2017) 385-403.
\bibitem{gaotemporal}G.-H. Gao, A. A. Alikhanov, Z.-Z. Sun, \textit{The temporal second order difference schemes based on the interpolation approximation for solving the time multi-term and distributed-order fractional sub-diffusion equations}, J. Sci. Comput. {\bf 73} (2017) 93-121.
\bibitem{alikhanov2015new}A. A. Alikhanov, \textit{A new difference scheme for the time fractional diffusion equation}, J. Comput. Phys. {\bf 280} (2015) 424-438.
\bibitem{sun2015finite}Z.-Z. Sun, G.-H. Gao, \textit{Finite Difference Methods for the Fractional Differential Equations}, Science Press, Beijing, 2015, (in Chinese).

\end{thebibliography}
\end{document}